\documentclass[12pt,draftcls,onecolumn]{IEEEtran}
\ifCLASSINFOpdf
\else
\fi
\usepackage{cite}
\usepackage{amsmath,amssymb,amsfonts}
\usepackage{graphicx}
\usepackage{textcomp}
\usepackage{subfigure}
\usepackage{epstopdf}
\usepackage{cases}
\usepackage[ruled,linesnumbered]{algorithm2e}
\usepackage{amsthm}

\usepackage{booktabs}
\usepackage{threeparttable}
\usepackage{multirow}
\usepackage{xcolor}
\usepackage{diagbox}
\usepackage[textsize=tiny]{todonotes}

\newcommand{\tabincell}[2]{\begin{tabular}{@{}#1@{}}#2\end{tabular}}
\newtheorem{theorem}{\textbf{Theorem}}
\newtheorem{lemma}{\textbf{Lemma}}
\newtheorem{assumption}{Assumption}

\newtheorem{remark}{Remark}

\begin{document}

\title{
Model-free False Data Injection Attack in Networked Control Systems: A Feedback Optimization Approach}
\author{Xiaoyu Luo$^\dag$, \IEEEmembership{Student Member, IEEE}, Chongrong Fang$^\dag$, \IEEEmembership{Member, IEEE}, Jianping He$^\dag$, \IEEEmembership{Senior Member, IEEE}, Chengcheng Zhao$^\ddag$, \IEEEmembership{Member, IEEE}, and Dario Paccagnan$^\S$, \IEEEmembership{Member, IEEE}
\thanks{$^\dag$: Department of Automation, Shanghai Jiao Tong University, and Key Laboratory of System Control and Information Processing, Ministry of Education of China, Shanghai 200240, China. E-mail: xyl.sjtu@sjtu.edu.cn, crfang@sjtu.edu.cn, jphe@sjtu.edu.cn.}
\thanks{$^\ddag$: State Key Laboratory of Industrial Control Technology and Institute of Cyberspace Research, Zhejiang University, China. E-mail: zccsq90@gmail.com.}
\thanks{$^\S$: Department of Computing, Imperial College London, London SW7 2AZ, UK. E-mail: d.paccagnan@imperial.ac.uk.}
}

\maketitle

\begin{abstract}
Security issues have gathered growing interest within the control systems community, as physical components and communication networks are increasingly vulnerable to cyber attacks. In this context, recent literature has studied increasingly sophisticated \emph{false data injection} attacks, with the aim to design mitigative measures that improve the systems' security. Notably, data-driven attack strategies -- whereby the system dynamics is oblivious to the adversary -- have received increasing attention. However, many of the existing works on the topic rely on the implicit assumption of linear system dynamics, significantly limiting their scope. Contrary to that, in this work we design and analyze \emph{truly} model-free false data injection attack that applies to general linear and nonlinear systems. More specifically, we aim at designing an injected signal that steers the output of the system toward a (maliciously chosen) trajectory. We do so by designing a zeroth-order feedback optimization policy and jointly use probing signals for real-time measurements. We then characterize the quality of the proposed model-free attack through its optimality gap, which is affected by the dimensions of the attack signal, the number of iterations performed, and the convergence rate of the system. Finally, we extend the proposed attack scheme to the systems with internal noise. Extensive simulations show the effectiveness of the proposed attack scheme.
\end{abstract}


\section{Introduction}
\IEEEPARstart{N}{etworked} control systems have seen a surge of interest in recent years, largely owing to their widespread applicability to commonly encountered problems including mobile robots coordination, smart grids operation, unmanned vehicles control, remote diagnosis, to mention but a few \cite{antsaklis2007special,gupta2009networked,zhang2019networked}.
As physical components and communication networks are increasingly vulnerable to cyber attacks, security issues have gathered growing traction in the community.

In this context, false data injection (FDI) -- whereby an attacker injects false data by compromising sensor readings or communication channels -- is a commonly encountered form of attack \cite{liu2011false}. Crucially, through an FDI attack, an adversary can cause significant damage to the infrastructure while remaining undetected. 

\subsection{Motivation}
Against this backdrop, recent literature has proposed increasingly sophisticated FDI attacks, with the hope that understanding their workings would lead to the design of mitigative measures that improve the systems' security \cite{chen2017optimal,zhao2020data,wang2021optimal,zhang2022design}. Among them, \emph{data-driven} attack strategies -- whereby an attack signal is designed solely relying on the available system's measurements -- have received growing attention. 
However, three critical issues deserve further consideration. First, many of the existing works tacitly assume that, while unknown, the underlying system follows a linear time invariant dynamics \cite{kim2014subspace,an2017data,Alisic23}. This significantly restricts the power of the adversary. Second, without any prior information about the system model, the ability that the adversary achieves its attack objective needs to be explored. It is conducive to analyzing the system’s vulnerability. Third, when the capability of the adversary is limited, e.g., the attack energy is limited \cite{zhang2015optimal}, it is practical and promising to excavate the potential attack's impact  while achieving the malicious objective. 

In this work, we tackle the aforementioned issues by proposing a novel data-driven FDI attack that, crucially, does not rely on prior information regarding the system's model, and applies to general non-linear dynamics. Towards this goal, we leverage zeroth-order optimization methods, a class of optimization algorithms that do not necessitate the availability of the cost function's gradient, but simply exploit function evaluations \cite{zhang2022new}. More specifically, we leverage zeroth-order optimization methods in the context of feedback optimization, where optimization algorithms are utilized as feedback controllers for dynamical systems \cite{Colombino20}. These tools provide a new approach to design data-driven attacks.

\subsection{Contributions}
In this work, we design a model-free FDI attack that does not rely on knowledge of the underlying dynamics and that applies to general linear \emph{and} non-linear systems. Specifically, we aim at steering the output of an unknown dynamical system to a (maliciously chosen) trajectory, through the sole use of real-time measurements. We do so in the bounded attack model, where an upper bound on the energy of the injected signal is given. A comparison between the existing FDI attack strategies and our work is shown in Table \ref{table1:comparison}. 


Compared to our conference version \cite{luo2022model-free}, we extend the proposed model-free attack strategy to systems with internal noise and explore its effects on the optimality of the obtained solutions. Moreover, we significantly expand upon the related work, motivation, performance analysis, and simulation results.
	The main contributions are summarized as follows.
			 	\begin{table}[h]
		\caption{Comparison of FDI attack designs}
		\label{table1:comparison}
		\vspace{-10pt}
		\begin{center}
			\begin{threeparttable}
			\begin{tabular}{c c c c}
				\specialrule{0.15em}{3pt}{3pt}
				  &
				\tabincell{c}{Model-based\\
				\cite{chen2017optimal,zhang2022design}, etc.}  & 
				 \tabincell{c}{Data-driven\\
				\cite{an2017data,zhao2020data}, etc.} &
				 \textbf{Our work}  \\
				
				\specialrule{0.05em}{3pt}{3pt}
				\textbf{\tabincell{c}{Prior \\Information}}  & \tabincell{c}{Knowledge \\of dynamics} & \tabincell{c}{Linear \\ Dyanamics}&None \\
				
				\specialrule{0.05em}{3pt}{3pt}
				\textbf{\tabincell{c}{Objective\\}}  & \tabincell{c}{Steer state\\ to desired point}  & \tabincell{c}{ Remain\\ undetected} & \tabincell{c}{Steer state to\\desired trajectory}\\

				\specialrule{0.05em}{3pt}{3pt}
				\textbf{\tabincell{c}{Approach\\}}  & \tabincell{c}{Dynamic\\ programming } & \tabincell{c}{Subspace \\method} &  \tabincell{c}{Zeroth-order\\ feedback-optimization}\\
				
				\specialrule{0.15em}{3pt}{3pt}
			\end{tabular}
	\end{threeparttable}
    \end{center}
		
	\end{table}
\begin{itemize}
		\item We construct a zeroth-order feedback optimization framework for the design of an FDI attack strategy, where the adversary has limited capability and no prior information about the system model.
		\item We propose a model-free attack scheme that drives the output of the system to a maliciously chosen trajectory. From a methodological standpoint, its novelty lies in directly updating the attack signal based on the objective function evaluations.
		\item We theoretically characterize the solution's optimality gap. Further, we analyze the impact of the attack signal's dimension, of the iteration numbers, of the variance of the objective function, and of the convergence rate of the dynamical system on the optimality gap. Finally, we extend the proposed model-free attack scheme to noisy systems and derive an upper bound of the optimality gap.
		
	\end{itemize}
 
\subsection{Paper organization}
	The rest of the paper is organized as follows. Section \ref{Related-work} reviews the related works. Section \ref{II} introduces the system and adversary model and formulates the FDI attack design problem. In Section \ref{III}, we design the proposed model-free attack strategy and analyze the optimality gap. Section \ref{secIV} extends the model-free attack scheme to noisy systems. In Section \ref{secV}, we analyze the design of stealthy attacks. Simulation results are presented in Section \ref{IV}. Finally, we conclude our work in Section \ref{V}.

\section{Related works}\label{Related-work}
Existing FDI attack strategies can be divided into two streams depending on whether they rely on a model-based or data-driven approach.

The literature on model-based FDI attack strategy is vast, and includes \cite{chen2017optimal,guo2018worst,wang2021optimal,zhang2014online,zhang2022design}. In the following, we review only the works that are most relevant to ours.
 When the adversary is aware of the system dynamics and other critical information (e.g., statistical properties of noise and the controller's feedback matrix), Chen \emph{et al.} \cite{chen2017optimal} formulate a linear quadratic cost function to steer the system state to a desired value, while satisfying a detection-avoidance constraint.
	In similar settings, i.e., when knowledge of the system dynamics is available, Guo \emph{et al.} \cite{guo2018worst} propose an innovation-based linear attack strategy and formulates a two-stage optimization problem to obtain the most-damaging attack policy. 
	In \cite{wang2021optimal}, Wang proposes an optimal attack strategy to deteriorate the performance of fault detectors by solving coupled backward recursive Riccati difference equations (RDEs). 
	In \cite{zhang2014online}, the authors design an FDI attack strategy against a remote state estimation algorithm with sensor-to-estimator communication rate constraint.
	With the knowledge of all system parameters except for the filter gain, Zhang \emph{et al.} \cite{zhang2022design} design stealthy attacks based on the Fisher information matrix to maximize the estimation error. 
	Note that the design of the above FDI attack strategies is mostly based on the full knowledge of the system model. However, when the system model changes or its exact knowledge is unavailable, the previous approaches can not be applied. 
	
	On the other hand, data-driven attack strategies have recently gained momentum \cite{esmalifalak2011stealth,kim2014subspace,an2017data,zhao2020data}. Two approaches are typically pursued. The first approach consists in exploiting offline observation of the system's dynamics to identify the matrices of the linear system model. Naturally, this approach does not apply to genuinely non-linear dynamics. The second approach consists in directly utilizing input-output data to design an attack strategy. For example, Esmalifalak \emph{et al.} \cite{esmalifalak2011stealth} apply linear independent component analysis (ICA) to estimate the system Jacobian matrix and design unobservable attacks based on the inferred structure. Kim \emph{et al.} \cite{kim2014subspace} extend the work in \cite{liu2011false} and present two data-driven attack strategies based on subspace methods. An \emph{et al.} \cite{an2017data} formulate the attack goal as a data-based $\mathcal{L}_2$-gain composite optimization problem and propose a new multiobjective adaptive dynamic programming (ADP) method for designing the attack policy. Zhao \emph{et al.} \cite{zhao2020data} propose an undetected FDI attack strategy based on a subspace identification technique to maximize the state estimation error. Note that the linearity of the system dynamics is still a crucial and implicit assumption necessary for all the aforementioned works. Our work relaxes this assumption and provides a new perspective to construct a completely model-free attack strategy based on the zeroth-order feedback optimization framework. 

\section{Problem formulation}\label{II}
 \subsection{System dynamic model $\&$ adversary model}
    Consider a discrete-time dynamical system
    \begin{align}\label{1}
    \begin{split}
    x_{k+1}=&f(x_k,u_k),\\
    y_{k}=&g(x_k),
    \end{split}
    \end{align}
	where $x_k\in \mathbb{R}^{n}$ is the system state at iteration $k$, $u_k\in \mathbb{R}^{m}$ is the system input, $y_{k}\in \mathbb{R}^{q}$ is the system output. 
	\begin{assumption}\label{ass1}
		The system (\ref{1}) is stable under the control of system input $u_k, \forall k \in \mathbb{N}$.
	\end{assumption}
 
 Consider that the adversary can compromise the stable system and manipulate the state $x_k$ arbitrarily and aims to steer the output value $y^a_k$ to its expected trajectory. The dynamical system under attack can be rewritten as
	\begin{align}\label{6}
	\begin{split}
	     x_{k+1}^{a}=&f(x_k^{a},u_k) + \Gamma \theta_{k}, \\
	     y_{k}^{a}=&g(x_k^{a}),
	\end{split}
	\end{align}
	where the attack selection matrix $\Gamma \in \mathbb{R}^{n\times p}$ is defined as the non-zero columns of $\mathrm{diag}(\gamma_1,\ldots,\gamma_n)$ with the binary variable $\gamma_i=1$ if the $i$-th dimensional state is compromised, and $\theta_k\in \mathbb{R}^{p}$ is the injected false data.
	Then, we make the following assumption about the ability of the adversary.
	\begin{assumption}\label{ass3}
		The capability of the adversary is limited, i.e., $\theta_k^{\mathrm{T}} \theta_k \leq R$, where $R$ is the upper bound of attack energy.
	\end{assumption}
 Assumption \ref{ass3} is common for energy-constrained adversaries \cite{zhang2015optimal}, which means that the injected false data is bounded. With Assumptions \ref{ass1} and \ref{ass3}, it is easy to obtain the following lemma to show that the compromised system (\ref{6}) is still stable.  
    \begin{lemma}\label{addass4}
    	For the compromised system (\ref{6}), there exists a unique steady-state map $x^a_{ss}: \mathbb{R}^{m}\times  \mathbb{R}^{p} \rightarrow \mathbb{R}^{n}$ such that $\forall \theta, f'(x^a_{ss}(u,\theta),u,\theta) \triangleq f(x^a_{ss}(u,\theta),u)+ \Gamma \theta=x^a_{ss}(u,\theta)$. The map $x^a_{ss}(u,\theta)$ is $M_x$-Lipschitz with respect to $\theta$, and the function $g(x^a)$ is $M_g$-Lipschitz with respect to $x^a$.
    \end{lemma}
\begin{remark}
Lemma \ref{addass4} is similar to \cite{he2022model} for guaranteeing the stability of the system. If the system under the bounded FDI attacks has no unique steady-state map $x^{a}_{ss}$, it is obvious that the system will diverse and even the original system (\ref{1}) is unstable. 
	The properties of the map $x^a_{ss}(u,\theta)$ can be ensured by the implicit function theorem \cite[Theorem 1B.1]{dontchev2009implicit}. 
	With Lemma \ref{addass4}, in the steady state we have
	\begin{align*}
	y^a=g(x^a_{ss}(u,\theta)) \triangleq h(u,\theta).
	\end{align*}
\end{remark}

	Additionally, the Lyapunov theorem presented in \cite[Theorem 2.7]{bof2018lyapunov} guarantees that there exist a Lyapunov function $V: \mathbb{R}^{n} \times \mathbb{R}^{m} \times \mathbb{R}^{p} \rightarrow \mathbb{R}$ and parameters $\alpha_1,\alpha_2,\alpha_3>0$ such that	
	\begin{align}
	\alpha_1 \|x^a-x^a_{ss}(u,\theta)\|^2\leq V(x^a,u,\theta) \leq \alpha_2 \|x^a-x^a_{ss}(u,\theta)\|^2, \label{3} \\
	V(f'(x^a_{ss}(u,\theta),u,\theta))-V(x^a,u,\theta)\leq -\alpha_3 \|x^a-x^a_{ss}(u,\theta)\|^2. \label{4}
	\end{align}
	Based on (\ref{3}) and (\ref{4}), the rate of the change in one step of the function value $V(x^a,u,\theta)$  is denoted by
	\begin{align}\label{a6}
	\mu \triangleq \frac{2\alpha_2}{\alpha_1} (1-\frac{\alpha_3}{\alpha_2}).
	\end{align}
	\begin{assumption}\label{ass2}
		The convergence rate $\mu$ satisfies $\mu<1$.
	\end{assumption}
 
	The smaller $\mu$ is, the faster the system converges to the steady state \cite{belgioioso2021sampled}. 
	The formal interpretation of $\mu$ will be presented later in Lemma \ref{l2}.
 
	\subsection{Problem formulation}
 	In this paper, we aim to design a completely model-free attack strategy, which is independent of the characteristics and parameters of the system model itself.
  
 Herein, we consider that the adversary's objective is to steer the output value $y^a_k$ to follow its expected malicious trajectory $\bar{y}_k$ as closely as possible. We also consider that the adversary has limited attack energy. Therefore, the total goal of adversaries is to reduce both the error between the true system output and expected trajectory and the consumed attack energy as much as possible. 
	In addition, since our proposed attack strategy performs the optimization with the same objective function at each iteration $k$, we omit the subscript $k$ and formally formulate the problem as
\begin{align}\label{problem}
\mathbf{\mathcal{P}_1}: \quad &\mathrm{min}_{\theta} ~\Phi(\theta,y^a)=\|y^a-\bar{y}\|+\theta^{\mathrm{T}} Q \theta \\
&\mathrm{s.t.}~  y^a=h(u,\theta), \nonumber \\
& \quad ~~ \theta^{\mathrm{T}} \theta \leq R, \nonumber
\end{align}
where $y^a=h(u,\theta)$ is the steady-state map under attacks in (\ref{6}) to guarantee the stability of the compromised system (\ref{6}), $\bar{y}$ is the expected trajectory and $Q\in \mathbb{R}^{p\times p}$ is the positive definite weight matrix chosen by the adversary according to the tradeoff between the limited attack energy and tracking deviation $\|y^a-\bar{y}\|$. We also make a common assumption for the optimized objective function as follows.
\begin{assumption}\label{ass4}
	The function $\Phi(\theta,y^a)$ is $M$-Lipschitz with respect to $\theta$, $M_y$-Lipschitz with respect to $y^a$, and $\inf_{\theta,y^a}\Phi(\theta,y^a)>-\infty$.
\end{assumption}

The challenges of solving problem $\mathcal{P}_1$ come from two aspects. One is the nonlinearity of the system model.
For the unknown nonlinear system model (\ref{6}), it is hard to regress its critical system parameters. The other is how to use the compromised measurements to guide the output value to move along the desired trajectory while reducing the consumed attack energy as much as possible. Since $h(u,\theta)$ is unknown, it is difficult to directly obtain the gradients of the objective function with respect to the independent variable $\theta$ to solve problem $\mathcal{P}_1$. 

The key idea of the zeroth-order optimization is to utilize the objective function evaluations to construct gradient estimates, thus avoiding using the gradients directly.
We aim to construct the gradient estimates of the objective function to solve problem $\mathcal{P}_1$. 
Different from the traditional zeroth-order optimization framework for the design of the controller with non-manipulated measurements, our design focuses on utilizing the compromised measurements to design the attack signal in the original control systems with designed controllers. Herein, we mainly explore the model-free attack strategy without detector constraints and the attack design under detector constraints will be analyzed in Section \ref{secV}.

\section{Model-free attack strategy design}\label{III}
    In this section, we first introduce the zeroth-order optimization framework, which is the basis of our attack strategy design. Then, we utilize real-time measurements to design the attack signal. Finally, we analyze the optimality of the proposed attack strategy.
    \subsection{Preliminaries of zeroth-order optimization}
     The attack strategy design in this paper is inspired by the gradient estimates based on the residual feedback in \cite{zhang2022new}. 
    
    For an objective function $\Phi(w):\mathbb{R}^{p} \rightarrow \mathbb{R}$, the gradient estimate proposed in \cite{zhang2022new} is
    \begin{align}\label{8}
    \hat{\nabla} \Phi(w_k)=\frac{v_k}{\delta}(\Phi(w_k+\delta v_k)-\Phi(w_{k-1}+\delta v_{k-1})),
    \end{align}
    where $v_k$ and $v_{k-1}$ are independent random vectors selected uniformly from the unit sphere $\mathcal{S}_{p} \triangleq \{v_k\in \mathbb{R}^{p}: \|v_k\|=1\}$, i.e., $v_k \sim U(\mathcal{S}_p) $ and $\delta>0$ is the smoothing parameter. Note that only a new objective function evaluation needs to be computed at each iteration in (\ref{8}) because the objective value evaluated at the previous iteration $k-1$ is reused at the current iteration $k$.

    According to \cite[Lemma 5]{zhang2022new}, $\hat{\nabla} \Phi(w_k)$ in (\ref{8}) is an unbiased estimate of the gradient of the smooth approximation $\Phi_{\delta}(w)$ for $\Phi(w)$ at $w_k$, where
    \begin{align}\label{9}
    \Phi_{\delta}(w)=\mathbb{E}_{v\sim U(\mathcal{S}_p)}[\Phi(w+\delta v)].
    \end{align}
    The properties of $ \Phi_{\delta}(w)$ are shown as follows.
    \begin{lemma}[\cite{he2022model}]
    	If $\Phi_{\delta}(w):\mathbb{R}^{p} \rightarrow \mathbb{R}$ is $M-$Lipschitz, then for any $w\in \mathbb{R}^{p}, \delta>0$ and $\Phi_{\delta}(w)$ defined in (\ref{9}), we have
    	\begin{subequations}
    	\begin{align}
    	 \mathbb{E}_{v \sim U(\mathbb{S}_p)} \lbrack \frac{p}{\delta} \Phi(w+\delta v)v \rbrack=& \nabla \Phi_{\delta}(w), \label{10a}\\
    	|\Phi_{\delta}(w) - \Phi(w)| \leq & M\delta, \label{10b}\\
    	\|\nabla \Phi_{\delta}(w) - \nabla \Phi(w) \|\leq & \frac{Mp}{\delta} \label{10c}.
    	\end{align}
    	\end{subequations}
    \end{lemma}
   From (\ref{10c}), we know that $\Phi_{\delta}(w)$ is $\frac{Mp}{\delta}$-smooth, i.e., its gradient $\nabla \Phi_{\delta}(w)$ is $\frac{Mp}{\delta}$-Lipschitz continuous.

   \begin{figure}[t]
	\centering
	\includegraphics[width=0.45\textwidth]{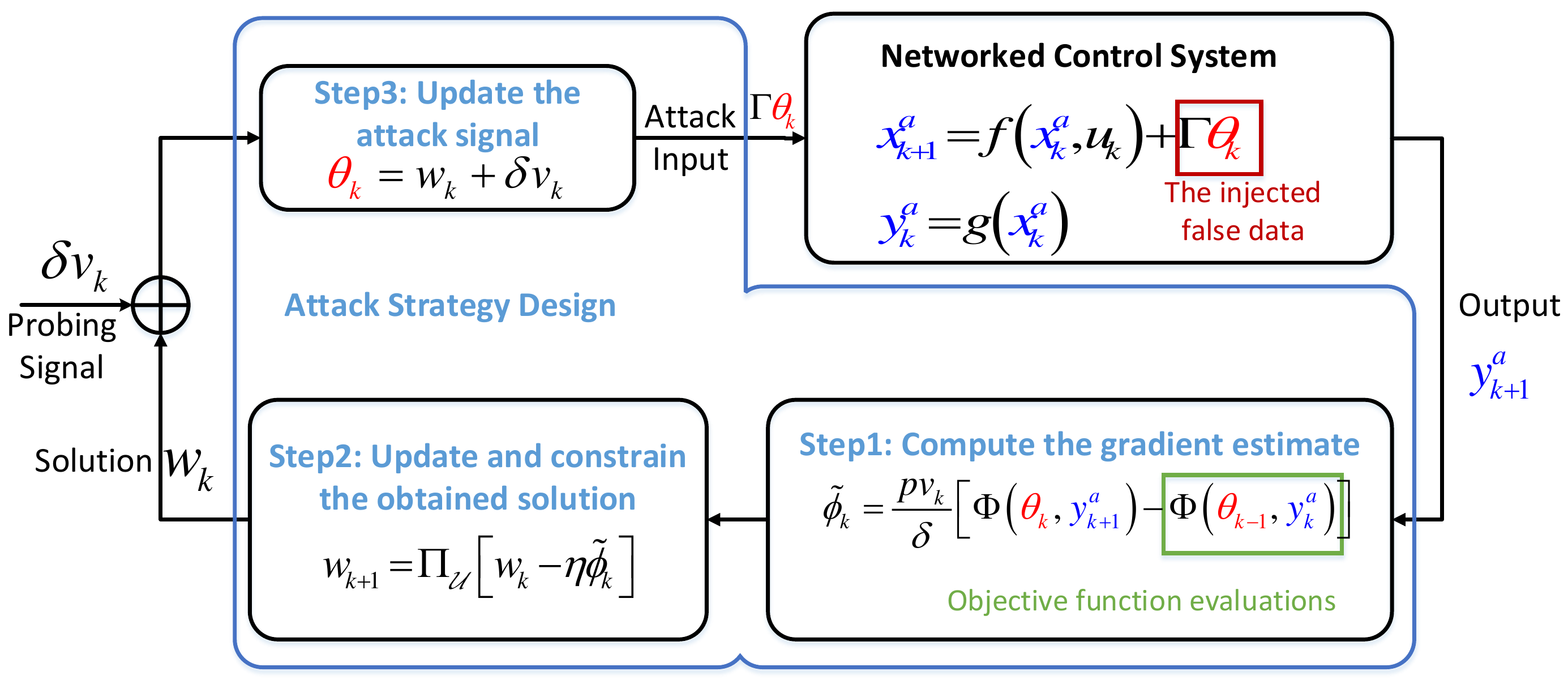}
	\caption{The schematic of model-free attack strategy design.}
	\label{schematic}
	\vspace*{-10pt}
\end{figure}
    \subsection{Attack strategy design}
    
    The proposed attack strategy iteratively updates attack inputs along the composite direction of the negative gradient estimates of the objective function and the projected gradients. Such a design only utilizes real-time measurements and thus makes the attack strategy intrinsically model-free. 
    
    We denote $\mathcal{U}$ as the constraint set in problem $\mathcal{P}_1$. With the zeroth-order optimization framework, the proposed model-free attack strategy can be divided into three steps and the schematic of the attack strategy design is shown in Fig.\ref{schematic}.  \\   
    \textbf{Step $1$: Compute the gradient estimate $\tilde{\phi}_k$}
    \begin{align}\label{add11}
    \tilde{\phi}_k=\frac{p v_k}{\delta} [\Phi(\theta_k,y_{k+1}^{a})- \Phi(\theta_{k-1},y_{k}^{a}) ],
    \end{align}
    where $v_k$ and $v_{k-1}$ are independent probing signals and follow the uniform distribution from the Euclidean unit sphere $\mathcal{S}_p$, i.e., $v_k\sim U(\mathcal{S}_p)$.
    Since only the real-time measurements are available for the adversary and it is hard to directly compute the gradients of the objective function in problem $\mathcal{P}_1$, we first utilize the probing signal $v_k$ for measurements, which can be used to construct the objective function evaluations $\Phi(\theta_k,y_{k+1}^{a})$ and $\Phi(\theta_{k-1},y_{k}^{a})$ at the current and previous iteration. Herein, the historic function evaluation $\Phi(\theta_{k-1},y^a_{k})$ is reused at iteration $k+1$. Then we compute the gradient estimates $\tilde{\phi}_k$ of the objective function by these evaluations with (\ref{add11}).\\
    \textbf{Step $2$: Update the obtained solution $w_{k+1}$}
    \begin{align}\label{add12}
    w_{k+1}=\Pi_{\mathcal{U}}[w_k - \eta \tilde{\phi}_k],
    \end{align}
    where $\Pi_{\mathcal{U}}[\cdot]$ is the projection onto constrained set $\mathcal{U}$, i.e., $\Pi_{\mathcal{U}}[l_1] \equiv \mathrm{arg} \min_{l_2\in \mathcal{U}} \|l_1-l_2\|$, and step-size $0<\eta<1$. To constrain the obtained solutions in the feasible region set by $\mathcal{U}$, we turn to the projected gradient descent method for updating the solution $w_{k+1}$ at iteration $k+1$ and solving the optimization problem $\mathcal{P}_1$ with constraints.\\
    \textbf{Step $3$: Update the attack signal $\theta_{k+1}$}
    \begin{align}\label{add13}
    \theta_{k+1}=w_{k+1} + \delta v_{k+1}.
    \end{align}
    Finally, the attack signal $\theta_{k+1}$ can be obtained by perturbing the solution $w_{k+1}$ with the probing signal $\delta v_{k+1}$.

     \subsection{Performance analysis}
    Let $\Phi(\theta_k) \triangleq \Phi(\theta_k,h(u_k,\theta_k))$. We use the optimality gap, i.e.,
    \begin{align}\label{12}
    \frac{1}{T} \sum_{k=1}^{T} \mathbb{E}_{v_{[T]}}[\Phi(\theta_k) - \Phi(\theta_k^*) ]
    \end{align}
    to measure the optimality of the proposed attack strategy at $\theta_k$ where $\theta^*_{k}$ is the optimal solution at iteration $k$ and $\mathbb{E}_{v_{[k]}}$ is the expectation with respect to $v_{[k]}$ where $v_{[k]} \triangleq (v_1,\ldots,v_k)$.
    
    Before we characterize (\ref{12}), we provide the upper bounds of $\|w_{k+1}-w_{k+1}^{*}\|^2$ and $\|w_{k+1}-w_k\|^2$, and some supporting lemmas for auxiliary analysis. We have
    \begin{align}\label{aa}
    \|w_{k+1}-w_{k+1}^{*}\|^2 =&\|\Pi_{\mathcal{U}}[w_k - \eta \tilde{\phi}_k]-w_{k+1}^{*}\|^2 \nonumber\\
    \overset{(s.1)}{\leq} &  \|w_k - \eta \tilde{\phi}_k-w_{k+1}^{*}\|^2 \nonumber\\
    \overset{(s.2)}{\leq} & 2 \|w_k-w_{k+1}^*\|^2 + 2 \eta^{2} \|\tilde{\phi}_k\|^2,
    \end{align}
    where $(s.1)$ follows from the projection property \cite[Lemma 2.4]{jain2017non} and \cite{nedic2018distributed}, i.e., for any $l_1\in \mathbb{R}^{p}$ and all $l_2\in \mathcal{U}$, we have $\|\Pi_{\mathcal{U}}[l_1]-l_2\|\leq \|l_1-l_2\|$, and $(s.2)$ follows the fact that $\|a-b\|^2 \leq 2(\|a\|^2 + \|b\|^2)$. Similarly, we have
    \begin{align}\label{a}
    \|w_{k+1}-w_k\|^2=&\|\Pi_{\mathcal{U}}[w_k - \eta \tilde{\phi}_k]-w_k\|^2 \nonumber \\   
    \leq &\|w_k - \eta \tilde{\phi}_k-w_k\|^2 \nonumber\\
    \leq & \eta^{2} \|\tilde{\phi}_k\|^2.
    \end{align}
    
    Note that we replace the steady output value $h(u_k,\theta_k)$ with the real-time output value $y^{a}_{k+1}$ to enter the closed-loop zeroth-order feedback optimization framework. It is unavoidable for the system to produce the error $e_{\Phi}(x^a_k,\theta_k)$, which is shown as 
    \begin{align}\label{add15}
    e_{\Phi}(x^a_k,\theta_k)=\Phi(\theta_k,y^a_{k+1})-\Phi(\theta_k,h(u_k,\theta_k)).
    \end{align}
    
   To derive the optimality gap \eqref{12}, we first analyze the upper bound of the error $e_{\Phi}(x^a_k,\theta_k)$ and recursive inequalities of two critical variables, i.e., $\mathbb{E}_{v_{[k]}}[V(x^a_k,u_k,\theta_k)]$ and $\mathbb{E}_{v_{[k]}}[\|\tilde{\phi}_k\|^2]$. 
    
    \begin{lemma}\label{l1}
    	If Assumptions \ref{ass1}, \ref{ass3}, \ref{ass2}, and \ref{ass4} hold, then we have
    	\begin{align}\label{14}
    	|e_{\Phi}(x^a_k,\theta_k)|^2\leq \frac{\mu M_{y}^2 M_{g}^2}{2 \alpha_2} V(x^a_k,u_k,\theta_k).
    	\end{align}
    \end{lemma}
\begin{lemma}\label{l2}
	If Assumptions \ref{ass1}, \ref{ass3}, \ref{ass2}, and \ref{ass4} hold, with (\ref{add11}), (\ref{add12}) and (\ref{add13}), then we have
	\begin{align}\label{add17}
	\mathbb{E}_{v_{[k]}}&[V(x^a_k,u_k,\theta_k)] \nonumber\\
	&\leq \mu \mathbb{E}_{v_{[k]}}[V(x^a_{k-1},u_{k-1},\theta_{k-1})] \nonumber \\
	&+ 4 \alpha_2 \eta^2 M_{x}^{2} \mathbb{E}_{v_{[k]}}[\|\tilde{\phi}_{k-1}\|^2] + 16 \alpha_2 \delta^2 M_{x}^{2}.
	\end{align}
\end{lemma}
 \begin{lemma}\label{l3}
	If Assumptions \ref{ass1}, \ref{ass3}, \ref{ass2}, and \ref{ass4} hold, with (\ref{add11}), (\ref{add12}), and (\ref{add13}), then we have
	\begin{align}\label{add18}
	\mathbb{E}_{v_{[k]}}[\|\tilde{\phi}_k\|^2] &\leq \frac{6 \eta^2 p^2 M^2}{\delta^2} \mathbb{E}_{v_{[k]}}[\|\tilde{\phi}_{k-1}\|^2]+24p^2M^2 \nonumber \\
	&+ \frac{3 \mu p^2 M_{y}^{2} M_{g}^{2} }{2 \alpha_2 \delta^2} (\mathbb{E}_{v_{[k]}}[V(x^a_k,u_k,\theta_{k})] \nonumber\\
	&+ \mathbb{E}_{v_{[k]}}[V(x^a_{k-1},u_{k-1},\theta_{k-1})]).
	\end{align}
\end{lemma}

The proofs of  Lemmas \ref{l1}, \ref{l2} and \ref{l3} are shown in Appendix \ref{A}, \ref{B} and \ref{C}, respectively. Lemma \ref{l1} quantifies the close relationship between $\Phi(\theta_k,y^a_{k+1})$ and $\Phi(\theta_k,h(u_k,\theta_k))$.  
Lemma \ref{l2} measures the proximity of the current state $x^a_k$ compared with the steady state $x^a_{ss}(u_k,\theta_k)$.   
    Lemma \ref{l3} reflects the first order smoothness of the objective function evaluation $\Phi(\theta_k,y^a_{k+1})$ at the solution $w_k$.

    Next, we provide the following theorem to characterize the optimality of the obtained solutions. Note that $\Phi(\theta_k) \triangleq \Phi(\theta_k,h(u_k,\theta_k))$. For simplicity, all the complexity results in this paper are presented in $\mathcal{O}$ notations.

    \begin{theorem}\label{optimality-gap}
	Supposing that Assumptions \ref{ass1}, \ref{ass3}, \ref{ass2}, and \ref{ass4} hold, for any given precision $\epsilon>0$ such that $|\Phi_{\delta}(\theta)-\Phi(\theta)|\leq \epsilon$, let $\delta=\frac{\epsilon}{M}$ and $\eta=\frac{\kappa \epsilon}{pT}$ with $0<\kappa<\kappa^*$, where
	\begin{align*}
	\kappa^*=\mathcal{O} \left(\min \left\{ \frac{T\sqrt{\mu(1+\mu)}}{\mu}, \frac{(1-\mu)T}{\sqrt{\mu(1+\mu)}} \right\} \right),
	\end{align*} 		
	then we have
	\begin{align}\label{20}
	\frac{1}{T} \sum_{k=1}^{T} &\mathbb{E}_{v_{[T]}}[\Phi(\theta_k) - \Phi(\theta_k^*) ] \nonumber\\
	&= \mathcal{O} \left(\frac{p^2(1+\mu)(1+\sqrt{1+\mu})}{(1-\rho)T^2} + \frac{\mu p^2}{T}\right),
	\end{align} where $\rho\in (0,1)$ is the maximum eigenvalue of matrix $P$ given by 
 \begin{align*}
     P=\left[  \begin{array}{cc}
     	p_{11} & \sqrt{p_{12}p_{21}}\\
     	\sqrt{p_{12}p_{21}} & p_{22}
     	\end{array}
     	\right]
 \end{align*} with
     	\begin{align*}
     	p_{11}=& \frac{6p^2 \eta^2}{\delta^2}(M^2 + \mu M_x^{2} M_y^{2} M_g^{2}), \quad p_{22}= \mu,\nonumber\\
     	p_{12}=& \frac{3\mu p^2 M_y^{2} M_g^{2}}{2 \alpha_2 \delta^2}(1+\mu), \quad p_{21}= 4\alpha_2 \eta^2 M_x^2, \nonumber \\
     	d_1=& 24p^2(M^2 + \mu M_x^{2} M_y^{2} M_g^{2}), \quad
     	d_2= 16\alpha_2 \delta^2 M_x^2.
     	\end{align*}
	Moreover, it also holds that
	\begin{align*}
	\rho=\mathcal{O} \left(\max \left\{ \frac{(1-\mu)^2}{1+\mu}, \mu\right\} +1-\mu\right).
	\end{align*} 
	
\end{theorem}
\begin{proof}
	Please see Appendix \ref{E}.
\end{proof}

	Theorem \ref{optimality-gap} shows that the optimality gap is related to the dimension $p$ of the attack signal, the convergence rate $\mu$ of the system, and the iterations $T$. As the iterations $T$ increase gradually, the optimality gap decreases and it can even decay to zero as long as $T$ is large enough. 
	
	\section{Noise effects on model-free attack design}\label{secIV}
	In this part, we further explore the effects of internal inherent noises on the proposed attack strategy and derive the optimality of solutions.
	\subsection{Problem reformulation}
 With noise $d_k$, the original system \eqref{6} can be rewritten as 
	\begin{align}\label{noise_attack_system}
	    \begin{split}
	   x_{k+1}^{a}=&f(x_k^{a},u_k,d_k) + \Gamma \theta_{k}, \\	    
	   y_{k}^{a,d}=&g(x_k^{a},d_k),
	    \end{split}
	\end{align}
	where the injected false data $\theta_k$ satisfies Assumption \ref{ass3} and the internal inherent noise $d_k\in \mathbb{R}^{r}$ is independent of the state $x_k$ and $\theta_k$ statistically. Herein, we consider the additive noise $d_k$, such as $f(x_k^{a},u_k,d_k)=f(x_k^{a},u_k)+d_k$. Similar to Assumption \ref{ass1}, the above discrete-time system is stable with noise $d_k$ before the invasion of attacks, which can be guaranteed by \cite[Theorem 2.2]{deng2001stabilization} if $d_k$ follows a standard Wiener process, i.e., the stochastic noise has zero mean and time-varying covariance. Let $\Phi(\theta,y^{a},d)=\|y^{a,d}-\bar{y}\|+\theta^{\mathrm{T}} Q \theta$. In this case, the optimization problem becomes
	\begin{align}\label{problem1}
\mathbf{\mathcal{P}_2}: \quad &\mathrm{min}_{\theta} ~\mathbb{E}_{d} [\Phi(\theta,y^{a},d)] \\
&\mathrm{s.t.}~  y^{a,d}=h(u,\theta,d), \nonumber \\
& \quad ~~ \theta^{\mathrm{T}} \theta \leq R, \nonumber
\end{align}
where $y^{a,d}=h(u,\theta,d)$ is the steady-state map under attacks in \eqref{noise_attack_system} to guarantee the stability of the compromised system. Let $\tilde{\Phi}(\theta) \triangleq \mathbb{E}_{d}[\Phi(\theta,y^{a},d)]$.
Then, we provide the following assumptions for the objective function $\Phi(\theta,y^{a},d)$.
\begin{assumption}\label{ass7}
For any $\theta\in \mathbb{R}^{p}$, there exists $\sigma>0$ such that
\begin{align}
    \mathbb{E}_{d}[(\Phi(\theta,y^{a},d)-\tilde{\Phi}(\theta))^2]\leq \sigma^2.
\end{align}
\end{assumption}
\begin{assumption}\label{ass8}
    The function $\Phi(\theta,y^{a},d)$ is $M(d)$-Lipschitz with respect to $\theta$, $M_y(d)$-Lipschitz with respect to $y^{a,d}$, and $\inf_{\theta,y^{a},d}\Phi(\theta,y^{a},d)>-\infty$. Moreover, we have $M(d)\leq M$ and $M_y(d)\leq M_y$.
\end{assumption}
Assumption \ref{ass7} provides a bounded variance of the objective function in the stochastic setting, which also implies that $\mathbb{E}_d [(\Phi(\theta,y^a,d_{01})-\Phi(\theta,y^a,d_{02}))^2]\leq 4\sigma^2$ \cite{zhang2022new}. In Assumption \ref{ass8}, the Lipschitz constants in the noisy system are constrained to be not larger than that in the noiseless system.
	
Moreover, the following lemma reveals that the compromised system \eqref{noise_attack_system} can still be stable in spite of the process and measurement noises and the noises do not influence the Lipschitz constant of the steady-state map.
	\begin{lemma}\label{l6}
	For the compromised system \eqref{noise_attack_system}, there exists a unique steady-state map $x^a_{ss}: \mathbb{R}^{m}\times  \mathbb{R}^{p} \times \mathbb{R}^{r}  \rightarrow \mathbb{R}^{n}$ such that $f'(x^a_{ss}(u,\theta,d),u,\theta,d) \triangleq f(x^a_{ss}(u,\theta,d),u,d)+ \Gamma \theta=x^a_{ss}(u,\theta,d)$ for any $\theta$. In addition, $x_{ss}^{a}(u,\theta,d)$ is $M_x$-Lipschitz with respect to $\theta$, and the function $g(x^a,d)$ is $M_g$-Lipschitz with respect to $x^a$.
	\end{lemma}
	\begin{proof}
	The proof can be divided into two parts. One is to find a Lyapunov function for guaranteeing the existence of the steady-state map. The other is to show the continuation property of the steady-state map based on the implicit function theorem \cite[Theorem 1B.1]{dontchev2009implicit}. 
	
	\textbf{Existence of the steady-state.} In the steady-state, we have
	\begin{align*}
	y^{a,d}=g(x^a_{ss}(u,\theta,d)) \triangleq h(u,\theta,d).
	\end{align*}
	Similarly, there exists the following Lyapunov function $V: \mathbb{R}^{n} \times \mathbb{R}^{m} \times \mathbb{R}^{p} \rightarrow \mathbb{R}$ and parameters $\alpha_1,\alpha_2,\alpha_3>0$ such that	
	{\small{
	\begin{align}
	c_1 \|x^a-x^a_{ss}(u,\theta,d)\|^2\leq V(x^a,u,\theta,d) \leq c_2 \|x^a-x^a_{ss}(u,\theta,d)\|^2, \label{add25} \\
	V(f'(x^a_{ss}(u,\theta,d)))-V(x^a,u,\theta,d)\leq -c_3 \|x^a-x^a_{ss}(u,\theta,d)\|^2. \label{addadd26}
	\end{align}}}
	Based on (\ref{add25}) and (\ref{addadd26}), the rate of the change in one step of the function value $V(x^a,u,\theta,d)$ is denoted as
	\begin{align}\label{a27}
	\mu^{\prime} \triangleq \frac{2c_2}{c_1} (1-\frac{c_3}{c_2}).
	\end{align}
	The stability of the compromised system can be guaranteed if $\mu^{\prime}<1$.
	
	\textbf{Continuity of the steady-state.} Let $F(x,u,\theta,d)=f(x^a_{ss}(u,\theta,d),u,d)+ \Gamma \theta-x^a_{ss}(u,\theta,d)=0$. Differentiating both sides of the above equation with respect to $\theta$ gives that
	\begin{align*}
    	    \frac{\partial{F}}{\partial{x_{ss}^{a}}} \frac{\partial{x_{ss}^{a}}}{\partial{\theta}}+ \frac{\partial{F}}{\partial{u}} \frac{\partial{u}}{\partial{x_{ss}^{a}}}\frac{\partial{x_{ss}^{a}}}{\partial{\theta}}+ \frac{\partial{F}}{\partial{\theta}} +  \frac{\partial{F}}{\partial{d}} \cdot 0=0.
	\end{align*}
	When $F(x,u,\theta,d)$ is continuously differentiable with respect to $\theta$ in the neighborhood of $(x_{ss}^{a},u,\theta,d)$ and $\frac{\partial{F}}{\partial{x_{ss}^{a}}}+\frac{\partial{F}}{\partial{u}} \frac{\partial{u}}{\partial{x_{ss}^{a}}}$ is nonsingular, $x_{ss}^a$ is the Lipschitz function with respect to $\theta$ where the Lipschitz constant $M_x$ satisfies
	\begin{align*}
	    M_x= \mathrm{sup } \left|- \frac{\frac{\partial{F}}{\partial{\theta}}}{\frac{\partial{F}}{\partial{x_{ss}^{a}}}+\frac{\partial{F}}{\partial{u}} \frac{\partial{u}}{\partial{x_{ss}^{a}}}}\right|. 
	\end{align*}
	Since we consider the additive noise from \eqref{noise_attack_system}, it can be followed that
	\begin{align*}
	    \frac{\partial{F(x,u,\theta)}}{\partial{\theta}} &= \frac{\partial{F(x,u,\theta,d)}}{\partial{\theta}},\\
	    \frac{\partial{F(x,u,\theta)}}{\partial{x_{ss}^{a}}} &=\frac{\partial{F(x,u,\theta,d)}}{\partial{x_{ss}^{a}}},
	\end{align*}
	where $F(x,u,\theta)=f(x^a_{ss}(u,\theta),u)+ \Gamma \theta-x^a_{ss}(u,\theta)=0$.  Thus, it is inferred that the existence of noise does not influence the Lipschitz continuous property of the steady-state with respect to $\theta$ and the Lipschitz constant is the same as that without noise. Similarly, for $F_y=g(x_{ss}^{a}(u,\theta,d))-y^{a,d}=0$, we have the same result. Hence, the proof is completed. 
	\end{proof}

 \begin{remark}\label{Re2}
	Lemma \ref{l6} is similar to Lemma \ref{l1} where the noise $d$ is independent of the state $x$ and the injected false data $\theta$. From Lemma \ref{l6}, we know that the process and measurement noises affect the convergence rate $\mu^{\prime}$ but not the Lipschitz constant of the steady-state map. Apparently, $\mu \neq \mu^{\prime}$. If $\mu^{\prime}>\mu$, the rate that the system converges to the steady state becomes slow (i.e., the noise reduces the convergence rate), which is shown in Fig. \ref{fig4} in Section \ref{IV}. 
	\end{remark}
 
\subsection{Attack strategy design with noise}
With the zeroth-order optimization framework, the model-free attack strategy under the discrete-time system with noise $d_k$ is designed as
\begin{align}\label{ss11}
\left\{
    \begin{aligned}
    &\tilde{\phi}_k=\frac{p v_k}{\delta} [\Phi(\theta_k,y_{k+1}^{a},d_k)- \Phi(\theta_{k-1},y_{k}^{a},d_{k-1}) ],\\
    &w_{k+1}=\Pi_{\mathcal{U}}[w_k - \eta \tilde{\phi}_k],\\
    &\theta_{k+1}=w_{k+1} + \delta v_{k+1},
    \end{aligned}
\right.
    \end{align}
    where $d_k$ and $d_{k-1}$ are independent random noises that are sampled at iterations $k$ and $k-1$, respectively. Different from \eqref{add11}, the existence of noise will also affect the objective function value. Moreover, the function value is not repeatable at different iterations and it is hard to store the noise value at each iteration for computing the function value. 
    Thus, at iteration $k$, only one evaluation is possible. In other words, compared to \eqref{add11}, it takes the residual of objective function evaluations between two consecutive stochastic feedback points.
    
 \subsection{Optimality with the general noise}
With the noise $d_k$, the following lemma provides the upper bound of $\mathbb{E}_{v_{[k]}}[V(x^a_k,u_k,\theta_k,d_k)]$ and $	\mathbb{E}_{v_{[k]}}[\|\tilde{\phi}_k\|^2]$ in this stochastic setting.
		\begin{lemma}\label{l8}
	If Assumptions \ref{ass3}, \ref{ass7} and \ref{ass8} hold, with \eqref{ss11}, we have
	\begin{align}\label{a30}
	\mathbb{E}_{v_{[k]}}&[V(x^a_k,u_k,\theta_k,d_k)] \nonumber\\
	&\leq \mu^{\prime} \mathbb{E}_{v_{[k]}}[V(x^a_{k-1},u_{k-1},\theta_{k-1},d_{k-1})]+ 32 c_2 \delta^2 M^{2}_{x} \nonumber \\
	&+ 8 c_2 \eta^2 M^{2}_{x} \mathbb{E}_{v_{[k]}}[\|\tilde{\phi}_{k-1}\|^2] + 16c_2 \sigma^2.
	\end{align}
\end{lemma}

 \begin{lemma}\label{l9}
	If Assumptions \ref{ass3}, \ref{ass7} and \ref{ass8} hold, with \eqref{ss11}, we have
	\begin{align}\label{add31}
	\mathbb{E}_{v_{[k]}}[\|\tilde{\phi}_k\|^2] &\leq \frac{12 \eta^2 p^2 M^2}{\delta^2} \mathbb{E}_{v_{[k]}}[\|\tilde{\phi}_{k-1}\|^2]+48p^2M^2 + \frac{24p^2 \sigma^2}{\delta^2}\nonumber \\
	&+ \frac{3 \mu^{\prime} p^2 M^{\prime^{2}}_{y} M^{\prime^{2}}_{g} }{2 c_2 \delta^2} (\mathbb{E}_{v_{[k]}}[V(x^a_k,u_k,\theta_{k},d_k)] \nonumber\\
	&+ \mathbb{E}_{v_{[k]}}[V(x^a_{k-1},u_{k-1},\theta_{k-1},d_{k-1})]).
	\end{align}
\end{lemma}
The proofs of Lemmas \ref{l8} and \ref{l9} are shown in Appendix \ref{F} and \ref{G}, respectively. Different from Lemmas \ref{l2} and \ref{l3}, the internal inherent noise leads to an additional term $16c_2 \sigma^2$ and $\frac{24p^2 \sigma^2}{\delta^2}$, respectively.

Next, we show the following theorem to characterize the effects of noise $d_k$ on the optimality of the obtained solutions.
\begin{theorem}\label{optimality_noise_gap}
Supposing that Assumptions \ref{ass3}, \ref{ass7} and \ref{ass8} hold, for any given precision $\epsilon >0$ such that $|\tilde{\Phi}_{\delta}(\theta)-\tilde{\Phi}(\theta)|\leq \epsilon$, let $\delta=\frac{\epsilon}{M}$ and $\eta=\frac{\kappa \sigma \epsilon}{p^2\sqrt{T}}$ with $0<\kappa<\frac{p^2\sqrt{T}}{\sigma \epsilon}$,
then we have
\begin{align}\label{aa20}
	\frac{1}{T} \sum_{k=1}^{T} &\mathbb{E}_{v_{[T]}}[\tilde{\Phi}(\theta_k) - \tilde{\Phi}(\theta_k^*) ] \nonumber\\
	&= \mathcal{O} \left(\frac{\sqrt{\mu^{^{\prime}}(1+\mu^{\prime)}}\sigma^3 p^3}{(1-\rho^{\prime})\sqrt{T}\epsilon^2} + \frac{\mu^{\prime} \sigma^2}{(1-\rho^{\prime})p^4}\right),
	\end{align}
	where $\rho^{\prime} \in (0,1)$ is the maximum eigenvalue of matrix $P^{\prime}$ given by \eqref{58}. 
\end{theorem}

\begin{remark}
The proof is shown in Appendix \ref{H}. As $T\rightarrow \infty$, the right side of \eqref{aa20} approaches $\frac{\mu^{\prime} \sigma^2}{(1-\rho^{\prime})p^4}$. The nonzero upper bound is related to the dimension $p$ of the injected false data, the variance of the objective function originating from noise and the convergence rate $\mu^{\prime}$. Compared with \eqref{20} in Theorem \ref{optimality-gap}, we also reveal that the existence of noise increases the optimality gap.
\end{remark}

	\section{Discussion}\label{secV}
	In this part, we show the detailed comparisons among the existing works on the design of the FDI attack strategy in Table \ref{table:comparison} and Table \ref{table:comparison1}, and introduce the feasible stealthy attack design. Since the design of the stealthy attack depends on the existence of the original attack detector, the produced stealthy attack strategy could be different due to distinct detection criteria. Moreover, the general assumption on the stealthy attack is that the knowledge of the existing detector is known. Herein, we discuss the following three detection criteria.

		 	\begin{table*}[h]
		\caption{The Comparisons among the Existing Works on Model-based attack strategy}
		\label{table:comparison}
		\vspace{-15pt}
		\begin{center}
			\begin{threeparttable}
			\begin{tabular}{c c c c c}
				\specialrule{0.15em}{3pt}{3pt}
				\textbf{Works}  &
				 \cite{guo2018worst} &
				 \cite{zhang2014online}   &
				 \cite{chen2017optimal}  & 
				 \cite{zhang2022design}  \\
				\specialrule{0.05em}{3pt}{3pt}
				\textbf{\tabincell{c}{System\\ Model}} & 
				 \tabincell{c}{$x_{k+1}=Ax_{k}+w_{k}$\\$y_{k}=Cx_{k}+v_{k}$}
				& \tabincell{c}{$x_{k+1}=Ax_{k}+w_{k}$\\$y_{k}=Cx_{k}+v_{k}$}
				&
				\multicolumn{2}{c}{\tabincell{c}{$x_{k+1}=Ax_{k}+Bu_{k}+D {\color{red}\zeta_{k}}+w_{k}$\\$y_{k}=Cx_{k}+E
				{\color{red}\beta_{k}} +v_{k}$}} \\
				\specialrule{0.05em}{3pt}{3pt}
			    \textbf{Noise}  & \multicolumn{4}{c}{Gaussian distribution with zero mean (i.i.d.)}\\

				\specialrule{0.05em}{3pt}{3pt}
				\textbf{\tabincell{c}{Stealthy \\Metric}} &  \tabincell{c}{Kullback-Leibler \\Divergence} & \tabincell{c}{Follow sensor-estimator \\ communicate rate} &$\chi^{2}$ detector &\tabincell{c}{Kullback-Leibler \\Divergence}\\
				
				\specialrule{0.05em}{3pt}{3pt}
				\textbf{\tabincell{c}{The \\Objective}}  &  \tabincell{c}{Max:  Estimation error \\ by stealthy linear attacks} & 
				\tabincell{c}{Degrade estimation quality \\ by stealthy attacks} &  \tabincell{c}{Track the desired  state \\ while keeping stealthy}  & \tabincell{c}{ Max: Estimation error \\by stealthy attacks}\\

				\specialrule{0.05em}{3pt}{3pt}
				\textbf{\tabincell{c}{The \\Methods}}  & Innovation-based & 
				Random theory & \tabincell{c}{Dynamic \\programming}  &  \tabincell{c}{Fisher \\information matrix}\\
				
				\specialrule{0.05em}{3pt}{3pt}
				\textbf{\tabincell{c}{The \\Requirements}}  & \multicolumn{3}{c}{All system model knowledge}  & \tabincell{c}{System parameters \\ without filter gain}\\
				
				\specialrule{0.15em}{3pt}{3pt}
			\end{tabular}
	\end{threeparttable}
    \end{center}
		
	\end{table*}

 If the detection criterion satisfies
	 \begin{align}
	     \|y^{a,d}_{k}-y^{d}_{k}\|\leq y_{\mathrm{th}},
	 \end{align}
	 the optimality of the obtained solutions in the proposed strategy remains as long as the actual output trajectory $y^{a,d}_{k}$ meets $ \|y^{a,d}_{k}-\bar{y}_{k}\|\leq 2y_{\mathrm{th}}$. Since it is a crude and inaccurate detection for a nonlinear/linear system, it is easy to deal with the stealthy constraint.
	 
	 If the detection criterion depends on the distribution gap between the normal output value and the compromised output value. For example, Kullback-Leibler divergence \cite{kullback1997information} is a good tool to measure how well two probability distributions match. Let $z_k=y^{d}_k-y_k$ and $z_k$ follows a known distribution. For example, in the linear system with Gaussian noise, the Kalman filter error $z_k$ is
	 an independent and identically distributed  (i.i.d) Gaussian variable with $z_k\sim \mathcal{N}(0,\Sigma)$. Let  $z^{a}_k=y^{a,d}_k-y_k$ and then the stealthy attacks should meet 
	 \begin{align}
	     D(z^{a}_k\|z_k)=\int_{\{\xi|f(z^{a}_k;\xi)>0\}} f(z^{a}_k;\xi) \mathrm{log} \frac{f(z^{a}_k;\xi)}{f(z_k;\xi)} \mathrm{d} \xi \leq \epsilon^{\prime},
	 \end{align}
	 where $\epsilon^{\prime}>0$ is a given stealthy parameter. 
	 In this case, the stealthy constraint can be further simplified when $z^{a}_k$ has the same statistical property as $z_k$.
	 
	 If the system adopts the data-driven detector, such as the machine-learning-based detection mechanisms \cite{maglaras2014intrusion,goh2017anomaly,anthi2021adversarial} or the behavior-based data-driven detection methods \cite{gadginmath2022direct}, the anomalies can be detected based on the characteristic of the chosen methods. Specifically,  the study \cite{maglaras2014intrusion} develops a One-Class Support Vector Machine (OCSVM) algorithm to classify the outlier class. The work \cite{goh2017anomaly} proposes the cumulative sum (CUSUM) method to detect the deviations that correspond to anomalies. In \cite{gadginmath2022direct}, a behavior-based $\chi^{2}$ detector was constructed based on a sequence of inputs and outputs and their covariance. 
	 When the stealthy attack is familiar with the existing learning-based/behavior-based detectors, the stealthy constraints can be derived and the obtained solutions are restricted in a new constraint set. Thus, the analysis of the updated constraint set is critical to the design of the FDI attack strategy with detectors.
			 		\begin{table*}[h]
		\caption{The Comparisons among the Existing Works on Data-driven attack strategy}
		\label{table:comparison1}
		\vspace{-15pt}
		\begin{center}
			\begin{threeparttable}
			\begin{tabular}{c c c c c c}
				\specialrule{0.15em}{3pt}{3pt}
				\textbf{Works}  &
				\cite{esmalifalak2011stealth}  & 
				 \cite{kim2014subspace} &
				 \cite{an2017data}    &
				 \cite{zhao2020data}  \\
				\specialrule{0.05em}{3pt}{3pt}
				\textbf{\tabincell{c}{System\\ Model}} & 
			\multicolumn{2}{c}{$z=Hx+e+{\color{red}a}$} & \multicolumn{2}{c}{\tabincell{c}{$x_{k+1}=Ax_{k}+Bu_{k}+D {\color{red}\zeta_{k}}+w_{k}$\\$y_{k}=Cx_{k}+E {\color{red}\beta_{k}} +w_{k}$}}\\
				\specialrule{0.05em}{3pt}{3pt}
			    \textbf{Noise}  & \multicolumn{2}{c}{Gaussian distribution} & $\sum_{k=0}^{T}\|w_k\|^{2}<\infty$& Gaussian distribution\\

				\specialrule{0.05em}{3pt}{3pt}
				\textbf{\tabincell{c}{Stealthy \\Metric}} & 	$z^{\prime}=H(x+a)+e$ & Subsapce $\mathcal{R}(H)$ & \tabincell{c}{$\alpha$-probability\\ $L_2$-stealthiness} &$\chi^2$ detector \\
				
				\specialrule{0.05em}{3pt}{3pt}
				\textbf{\tabincell{c}{The \\Objective}}  & \tabincell{c}{Design stealthy\\ FDI attack}  & \tabincell{c}{Design feasible \\unobservable attack} & 
			\tabincell{c}{Max: Stealthy \\attack's effects} &  \tabincell{c}{ Design data-driven \\ undetected attacks}\\

				\specialrule{0.05em}{3pt}{3pt}
				\textbf{\tabincell{c}{The \\Methods}}  & \tabincell{c}{Independent component \\analysis (ICA)} & Subspace method & 
				\tabincell{c}{Adaptive dynamic \\programming (ADP)} &  Subspace method\\
				
				\specialrule{0.05em}{3pt}{3pt}
				\textbf{\tabincell{c}{The \\Requirements}}  & \tabincell{c}{Independent, non-Gaussian \\ and the full sensors observations}  & \tabincell{c}{Linear  \\measurement model} & 
				\tabincell{c}{The attack \\strategy is linear} &  \tabincell{c}{Linear \\system model}\\
				
				\specialrule{0.15em}{3pt}{3pt}
			\end{tabular}
	\end{threeparttable}
    \end{center}
		
	\end{table*}
 
	\section{Simulation results}\label{IV}
    In this section, we evaluate the performance of the proposed attack scheme, i.e., the tracking performance and the optimality of solutions without/with noise.
    
    Consider the following system
    \begin{align}\label{linearmodel}
    \begin{split}
    x_{k+1}=&Ax_{k} + B u_k + d_k,\\
    y_{k}=&C x_{k} +d_k.
    \end{split}
    \end{align}
    where $u_k=-K x_k$ with $K=[1.5~ -1.5;0.2~ 0.1], A=[0~ 1;2~ -1], B=[0~ 0;1~ 0], C=[1~ 1]$. It is stable, controllable, and observable. We consider two kinds of noise, including $d_k\sim \mathcal{N}(0,0.02)$ and $d_k\sim U(0,0.02)$. We set the initial state $x_1=[1;-3]$, the probing signal $v_k=[\mathrm{cos}(k);\mathrm{sin}(k)]/\sqrt{2}$ to satisfy $\|v_k\|=1$, and the initial solution $w_1$ is random and follows the standard uniform distribution. We also set the smoothing parameter $\delta=10^{-3}$, the step-size $\eta=7.5\times 10^{-5}$, the attack selection matrix $\Gamma=I_2$ and the weight matrix $Q=3I_2$ where $I_2$ is a two-dimensional diagonal unit matrix. We define two types of the expected output trajectories, including the static trajectory $\bar{y}_1=-1.5$ and dynamic output trajectory $\bar{y}_2=10^{-4}k$ with respect to iteration $k$. 
    Each data point in the following figures represents an ensemble average of $50$ trials.
    
    Without noise $d_k$, we first analyze the tracking performance with different desired output trajectories. As shown in Fig. \ref{fig1}, the output value of the system under the proposed attack strategy has the ability to track the expected output trajectory whether the trajectory is static or dynamic. Especially, Fig. \ref{fig1a} and Fig. \ref{fig1b} illustrate that the output values fluctuate along the desired trajectory. Note that the phenomenon of fluctuation is normal since the output values are constantly perturbed by the time-varying probing signal $v_k$.      
 
    Then, we illustrate the optimality of solutions via the optimality gap $\Phi(\theta_k)-\Phi(\theta^*)$, which is shown in Fig. \ref{fig2}. When the expected trajectory is static, i.e., $\bar{y}_1=-1.5$, we find that the obtained solution is close to the optimal solution and the optimality gap converges to about $0.02$, as shown in Fig. \ref{fig2a}. When the expected trajectory is time-varying, i.e., $\bar{y}_2=10^{-4}k$, in Fig. \ref{fig2b}, the obtained solutions also approach the optimal one and the upper bound of the optimality gap does not exceed $0.11$. To sum up, the proposed model-free attack strategy can obtain the suboptimal attack signals that drive the output values to the desired output trajectory by only utilizing the real-time compromised measurements.
    
    \begin{figure}[t]
    	\centering
    	\subfigure[]{\label{fig1a}
    		\includegraphics[width=0.35\textwidth]{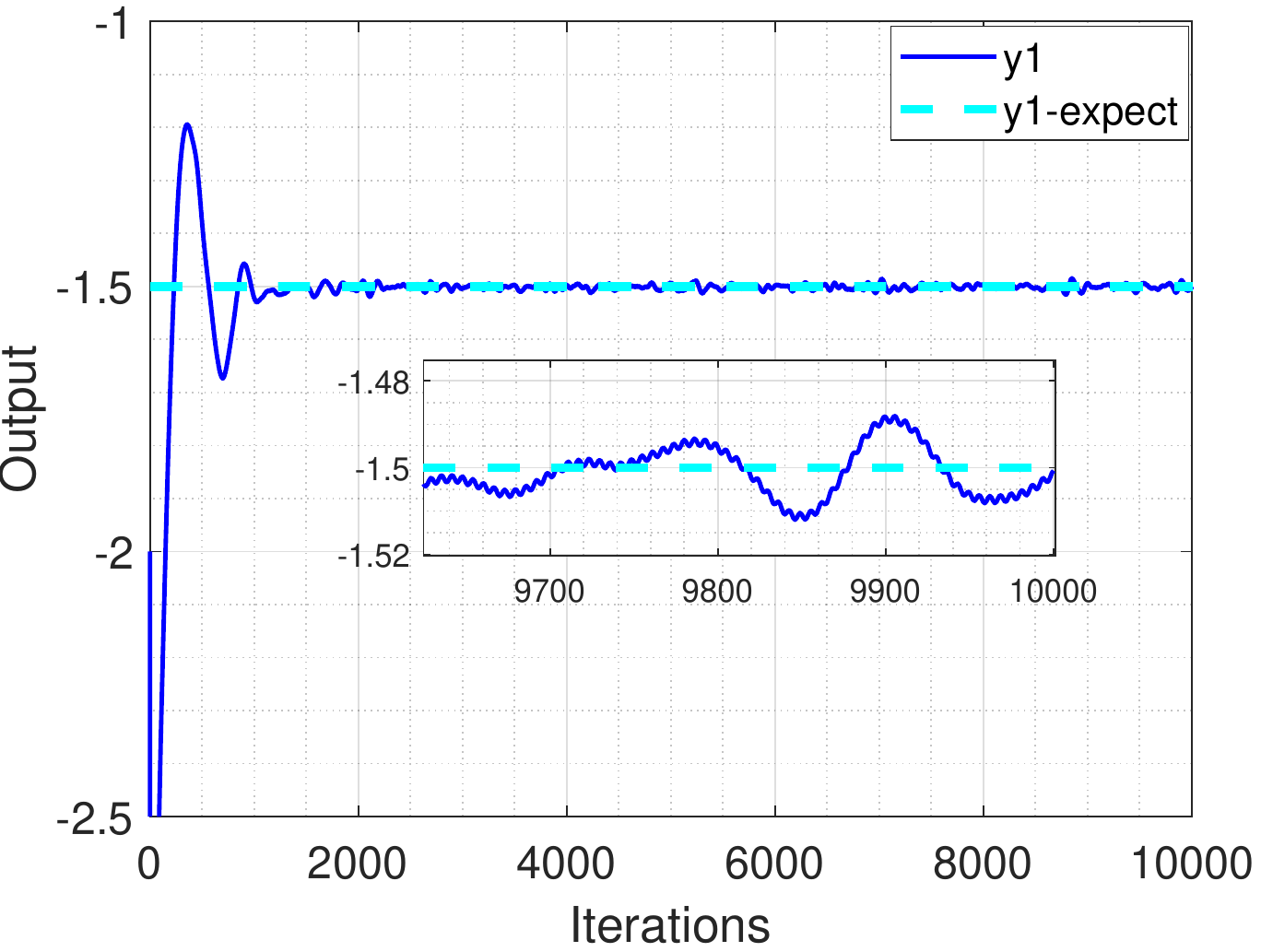}
    	}	
    	\subfigure[]{\label{fig1b}
    		\includegraphics[width=0.35\textwidth]{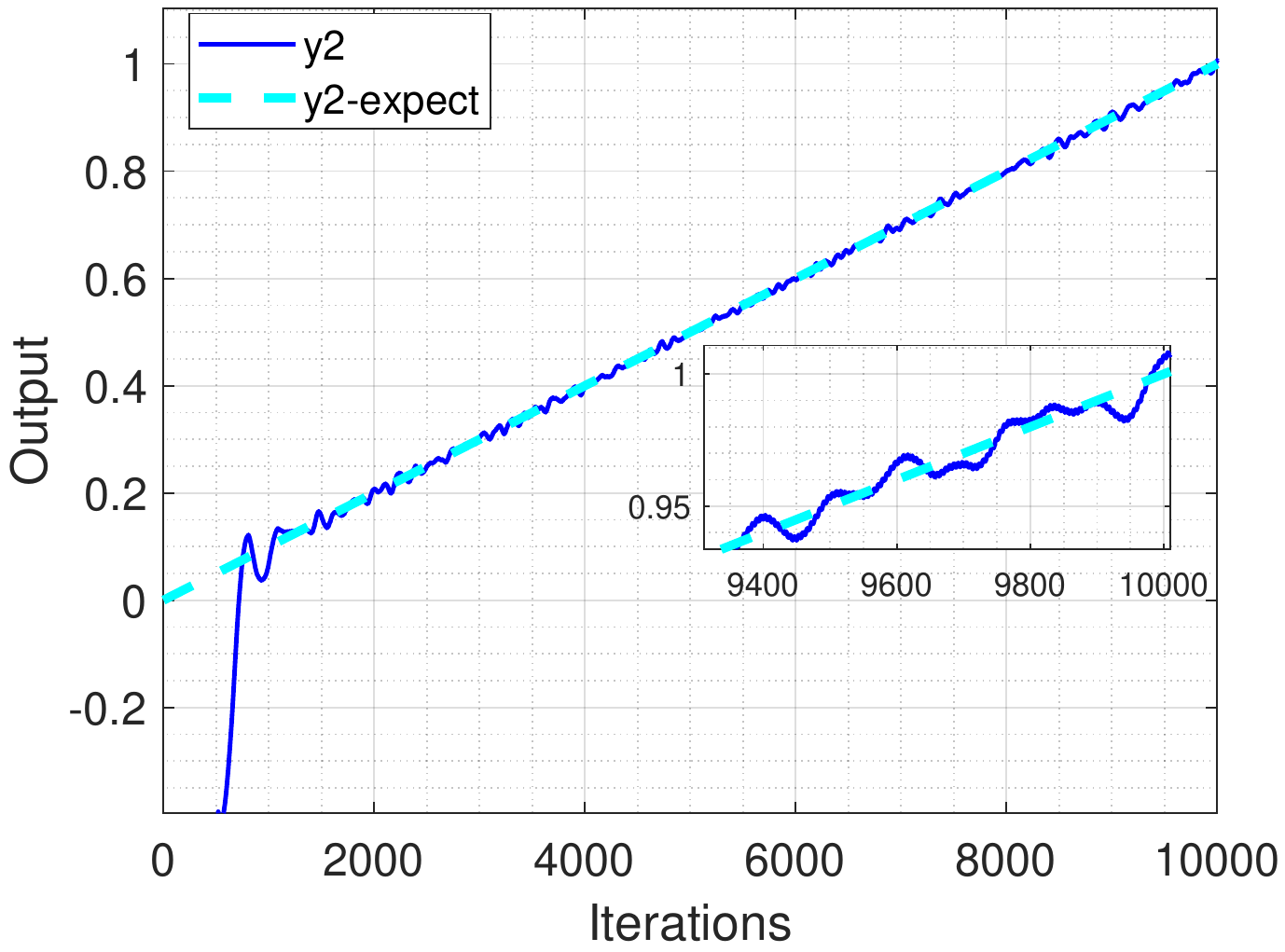}
    	}
    	\caption{Tracking performance under different expected output trajectories. (a) Static trajectory $\bar{y}_1=-1.5$ (b) Dynamic trajectory $\bar{y}_2=10^{-4}k$}
    	
    	\label{fig1}
    	\vspace*{-10pt}
    \end{figure}

\begin{figure}[t]
	\centering
	\subfigure[]{\label{fig2a}
		\includegraphics[width=0.35\textwidth]{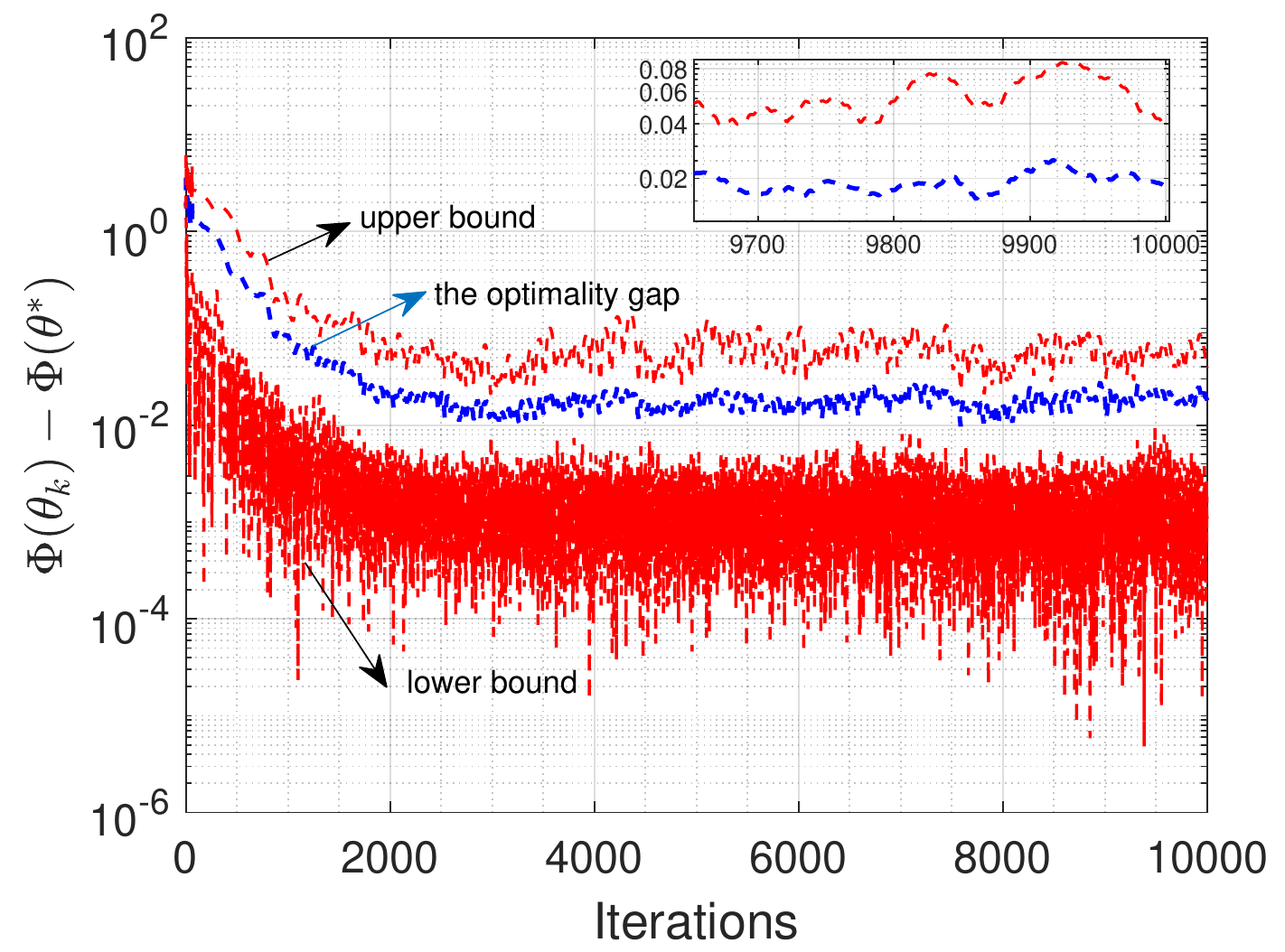}
	}	
	\subfigure[]{\label{fig2b}
		\includegraphics[width=0.35\textwidth]{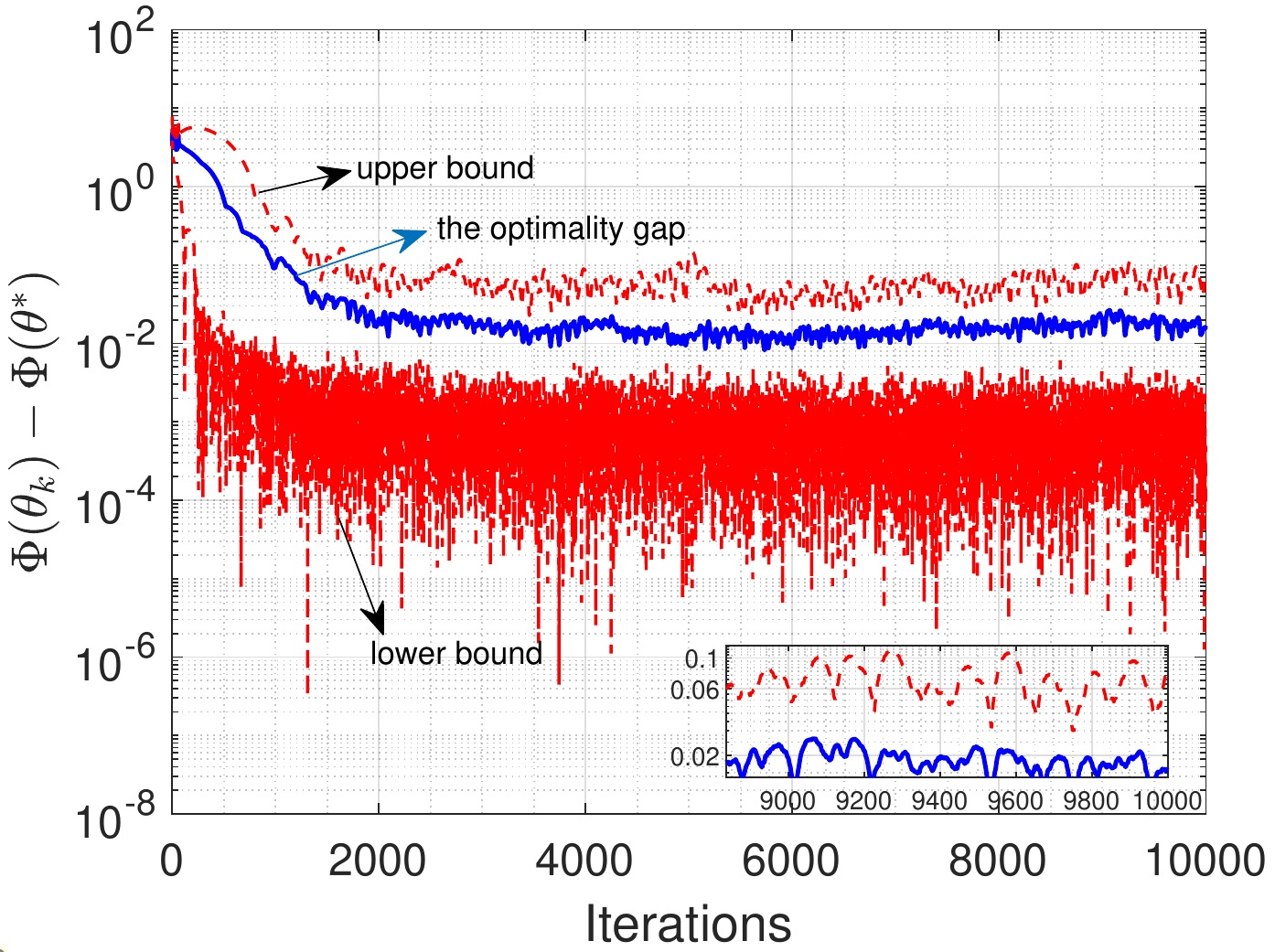}
	}
	\caption{The optimality gap under different expected output trajectories. (a) Static trajectory $\bar{y}_1=-1.5$ (b) Dynamic trajectory $\bar{y}_2=10^{-4}k$}
	
	\label{fig2}
	\vspace*{-10pt}
\end{figure}

With noise $d_k$, we analyze its effects on the tracking performance and optimality. The output $\mathrm{y2}$ and the optimality gap $\tilde{\Phi}(\theta_k)-\tilde{\Phi}(\theta^*)$ under uniform distribution noise and normal distribution noise are denoted as $\mathrm{y2-U}$, $\mathrm{y2-N}$, $\mathrm{Phi-U}$ and $\mathrm{Phi-N}$, respectively.
    From boxplot Fig. \ref{fig3} with iterations $k=40000$, we know that the final value of the actual output $y2$ is $4$ and the median is $2$. In other words, the slope of the dynamic trajectory of the output value is $10^{-4}$, which follows the expected one, and the noise does not influence the tracking trend while adding lots of outliers. In addition, combined with Fig. \ref{fig4}, the average optimality gap (red line in Fig. \ref{fig3} / blue line in Fig \ref{fig4}) of $50$ trials approaches zero although there are some outliers (red plus in Fig. \ref{fig3} / pink shadow in Fig. \ref{fig4}). Moreover, the optimality gap is larger than that without noise $d_k$ and 
    the normal distribution noise has smaller effects than the uniform distribution noise on the optimality gap.
    \begin{figure}[t]
	\centering
	\includegraphics[width=0.35\textwidth]{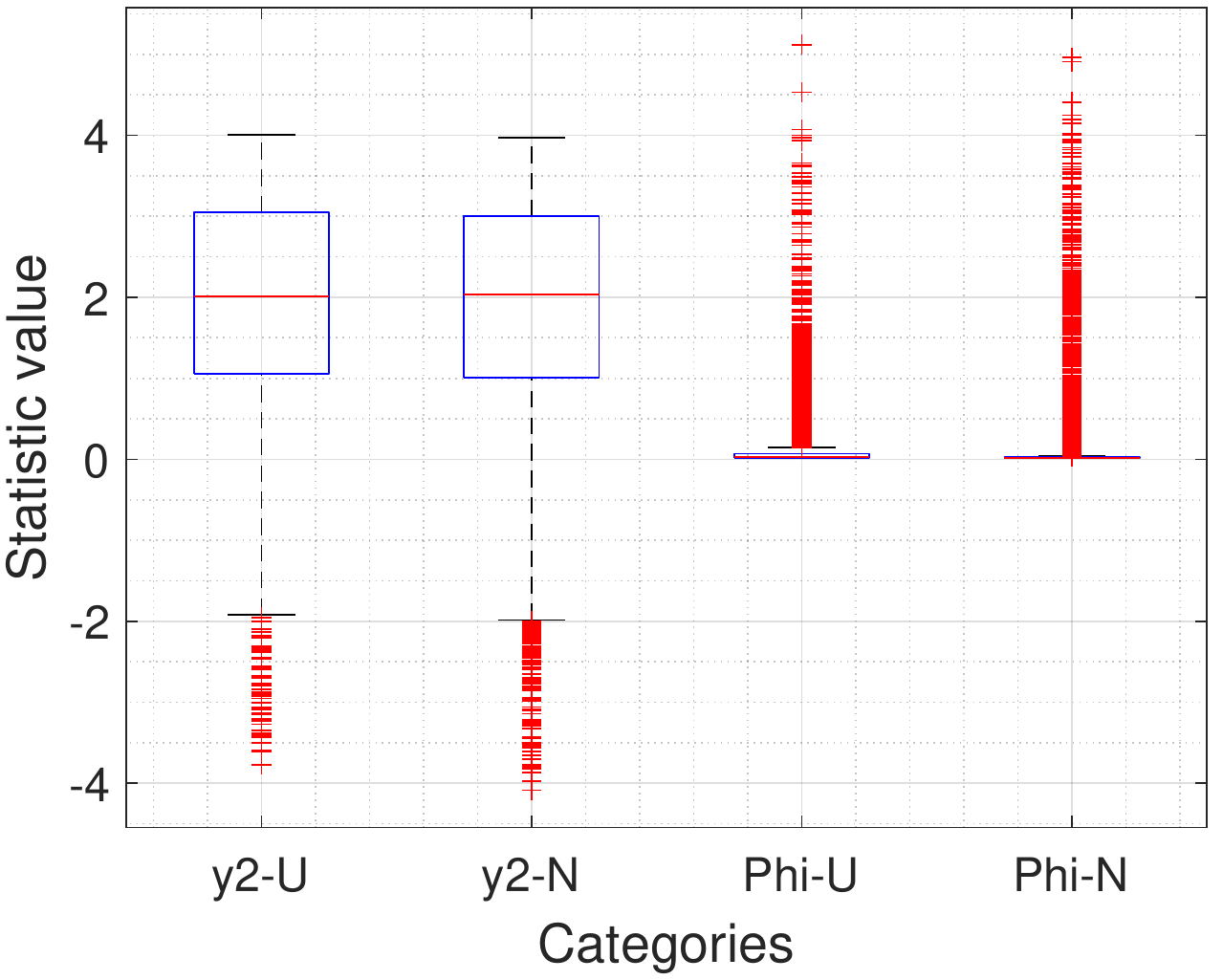}
	\caption{Comparison of the output and the optimality gap under noise.}
	\label{fig3}
	\vspace*{-10pt}
\end{figure}

\begin{figure}[t]
	\centering
    \includegraphics[width=0.35\textwidth]{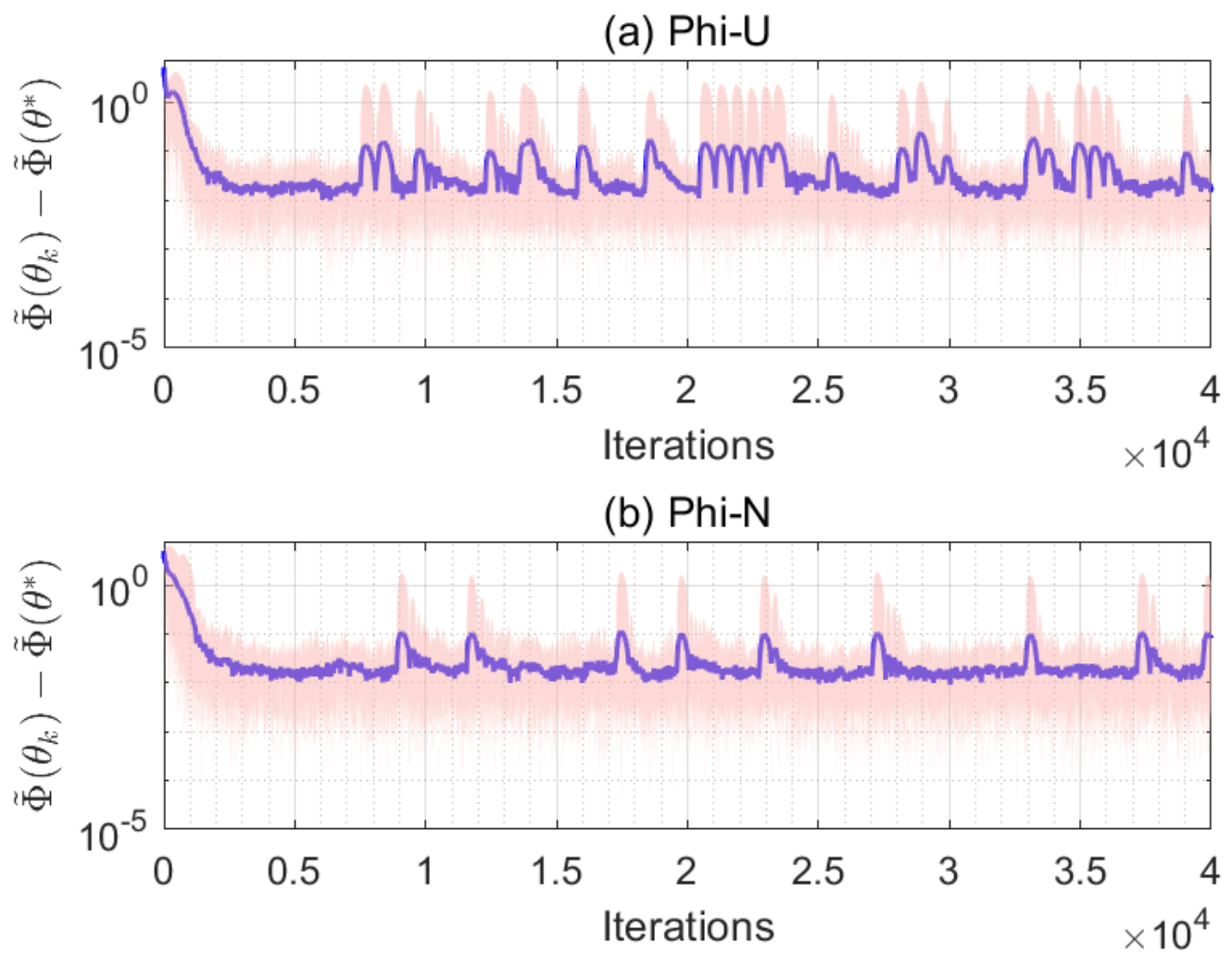}
	\caption{The bound of the optimality gap under noise.}
	\label{fig4}
	\vspace*{-10pt}
\end{figure}

\section{Conclusion}\label{V}
We considered the problem of designing a model-free attack scheme where the adversary with limited capability aims to make the output value follow the desired trajectory without any prior system model information. The designed attack scheme is model-free since only real-time measurements are required. These measurements are used to compute objective function evaluations and gradient estimates are constructed to update the attack signal based on these objective function evaluations at the previous and current time. Moreover, considering the adversary has limited capability, we constrained the obtained solutions within the feasible region by the projected gradient descent method. Finally, we analyzed the optimality of solutions and established its dependence on the dimensions of the attack signal, the iterations, the variance of the objective function, and the convergence rate of the system. Future works include the design of attack strategies with partial observations and specific detector constraints.   
 \section{APPENDIX}
    \subsection{Proof of Lemma \ref{l1}}\label{A}
Based on Assumptions $1$-$4$, then we have
		\begin{align}\label{a1}
		|e_{\Phi}(x_k,\theta_k)|^2=&|\Phi(\theta_k,y^a_{k+1})-\Phi(\theta_k,h(u_k,\theta_k))|^2 \nonumber \\
		\leq & M_{y}^2 \|y^a_{k+1} - h(u_k,\theta_k)\|^2 \nonumber \\
		\leq & M_{y}^2 M_{g}^2 \|x^a_{k+1}-x^a_{ss}(u_k,\theta_k)\|^2.
		\end{align}
		Combing (\ref{3}) with (\ref{4}), it can be inferred that
		\begin{align}\label{a2}
		\|x^a_{k+1}&-x^a_{ss}(u_k,\theta_k)\|^2 \nonumber \\
		&\leq  \frac{1}{\alpha_1} V(x^a_{k+1},u_k,\theta_k)= \frac{1}{\alpha_1} V(f'(x^a_{k},u_k,\theta_k),u_k,\theta_k)\nonumber \\
		&\leq  \frac{1}{\alpha_1} ( V(x^a_k,u_k,\theta_k) - \alpha_3 \|x^a_k - x^a_{ss}(u_k,\theta_k)\|^2)\nonumber \\
		&\leq  \frac{1}{\alpha_1} (1-\frac{\alpha_3}{\alpha_2}) V(x^a_k,u_k,\theta_k).
		\end{align}
		Substituting (\ref{a2}) into (\ref{a1}), it is easy to obtain (\ref{14}). 
	
	\subsection{Proof of Lemma \ref{l2}}\label{B}
		Based on (\ref{3}), we have
		{\small{
  \begin{align*}
		&V(x^a_k,u_k,\theta_k)\leq  \alpha_2 \|x^a_k - x^a_{ss}(u_k,\theta_k)\|^2 \nonumber \\
		&= \alpha_2 \|x^a_k - x^a_{ss}(u_{k-1},\theta_{k-1}) + x^a_{ss}(u_{k-1},\theta_{k-1})- x^a_{ss}(u_k,\theta_k) \|^2 \nonumber \\
		&\overset{(s.1)}{\leq} 2\alpha_2 (\|x^a_k - x^a_{ss}(u_{k-1},\theta_{k-1})\|^2 + \|x^a_{ss}(u_{k-1},\theta_{k-1})- x^a_{ss}(u_k,\theta_k) \|^2) \nonumber \\
		&\overset{(s.2)}{\leq} 2\alpha_2 (\frac{1}{\alpha_1}(1-\frac{\alpha_3}{\alpha_2} V(x^a_{k-1},u_{k-1},\theta_{k-1})) + M_{x}^{2} \|\theta_k-\theta_{k-1}\|^2),
		\end{align*}}}where $(s.1)$ follows the fact that $\|a+b\|^2 \leq 2(\|a\|^2 + \|b\|^2)$ and $(s.2)$ follows from \eqref{6}, \eqref{a2}, and the Lipschitz continuity of $x^a_{ss}(u_k,\theta_k)$. The upper bound of $\mathbb{E}_{v_{[k]}}[\|\theta_k-\theta_{k-1}\|^2]$ is given as
		\begin{align*}
		\mathbb{E}_{v_{[k]}}&[\|\theta_k-\theta_{k-1}\|^2] \nonumber \\
		&= \mathbb{E}_{v_{[k]}} [\| w_k- w_{k-1} + \delta v_k - \delta v_{k-1} \|^2] \nonumber \\
		&\overset{(s.1)}{\leq}  \mathbb{E}_{v_{[k]}} [2 \|w_k - w_{k-1}\|^2 + 2 \delta^2 \|v_k - v_{k-1}\|^2]\nonumber \\
		&\overset{(s.2)}{\leq}  2\eta^2 \mathbb{E}_{v_{[k]}}[\|\tilde{\phi}_{k-1}\|^2] + 2 \delta^2 \mathbb{E}_{v_{[k]}} [2 \|v_k\|^2 + 2\|v_{k-1}\|^2]\nonumber \\
		&\overset{(s.3)}{\leq}  2\eta^2 \mathbb{E}_{v_{[k]}}[\|\tilde{\phi}_{k-1}\|^2] + 8 \delta^2,
		\end{align*}
		where $(s.1)$ follows that $\mathbb{E}[(a+b)^2] \leq 2 \mathbb{E}[a^2+b^2]$, $(s.2)$ follows from (\ref{a}) and $\|a-b\|^2 \leq 2(\|a\|^2 + \|b\|^2)$, and $(s.3)$ follows the fact that $\|v_k\|=1$ since $v_k$ is selected uniformly at random from the unit sphere.
	
		Combining the above results, we can infer that (\ref{add17}) holds.

	\subsection{Proof of Lemma \ref{l3}}\label{C}
	Let $\Phi(\theta_k) \triangleq \Phi(\theta_k,h(u_k,\theta_k))$ and $\Phi(\theta_{k-1}) \triangleq \Phi(\theta_{k-1},h(u_{k-1},\theta_{k-1}))$. With (\ref{add11}) and (\ref{add15}), then we have
		\begin{align*}
		&\mathbb{E}_{v_{[k]}}[\|\tilde{\phi}_k\|^2] \nonumber \\
		&=\frac{p^2}{\delta^2} \mathbb{E}_{v_{[k]}}[\|v_k (\Phi(\theta_k,y^a_{k+1})- \Phi(\theta_{k-1},y^a_{k}))\|^{2}] \nonumber\\
		&=  \frac{p^2}{\delta^2} \mathbb{E}_{v_{[k]}}[\|v_k (\Phi(\theta_k)- \Phi(\theta_{k-1}) \nonumber\\
		&+ e_{\Phi}(x^a_k,\theta_k) - e_{\Phi}(x^a_{k-1},\theta_{k-1}) )\|^{2}] \nonumber\\
		&\overset{(s.1)}{\leq}  \underbrace{\frac{3p^2}{\delta^2} \mathbb{E}_{v_{[k]}}[\|v_k (\Phi(\theta_k)- \Phi(\theta_{k-1}))\|^{2}]}_{\textcircled{1}} \nonumber\\
		&+ \underbrace{\frac{3p^2}{\delta^2} \mathbb{E}_{v_{[k]}}[\|v_k e_{\Phi}(x^a_k,\theta_k)\|^{2}]
		+ \frac{3p^2}{\delta^2} \mathbb{E}_{v_{[k]}}[\|v_k e_{\Phi}(x^a_{k-1},\theta_{k-1})\|^{2}]}_{\textcircled{2}},
		\end{align*}
		where $(s.1)$ follows the fact that $\mathbb{E}[(a+b+c)^2] \leq 3 \mathbb{E}[a^2+b^2+c^2]$. Next, we provide the upper bound of the item $\textcircled{1}$ and $\textcircled{2}$, respectively.
		\begin{align}\label{add1}
		\textcircled{1} \overset{(s.1)}{\leq}& \frac{3p^2}{\delta^2} \mathbb{E}_{v_{[k]}}[\|v_k (\Phi(w_k + \delta v_k)- \Phi(w_{k-1} + \delta v_k)) \nonumber\\
		&+ \Phi(w_{k-1} + \delta v_k) - \Phi(w_{k-1} + \delta v_{k-1})\|^2] \nonumber\\
		=& \frac{6p^2}{\delta^2} \mathbb{E}_{v_{[k]}}[\|v_k (\Phi(w_k + \delta v_k)- \Phi(w_{k-1} + \delta v_k)) \|^2  \nonumber \\
		&+ \|\Phi(w_{k-1} + \delta v_k) - \Phi(w_{k-1} + \delta v_{k-1})\|^2] \nonumber \\
		\overset{(s.2)}{\leq}&  \frac{6p^2M^2}{\delta^2} \mathbb{E}_{v_{[k]}}[\|v_k\|^2 \|w_k - w_{k-1}\|^2 \nonumber \\
		&+ \|v_k\|^2 \delta^2 \|v_k - v_{k-1}\|^2] \nonumber \\
		\overset{(s.3)}{\leq} & \frac{6p^2M^2}{\delta^2} (\eta^2 \mathbb{E}_{v_{[k]}}[\|\tilde{\phi}_{k-1}\|^2] \nonumber \\
		&+ \delta^2 \mathbb{E}_{v_{[k]}}[\|v_k\|^2 (2 \|v_k\|^2 + 2 \|v_{k-1}\|^2)]) \nonumber \\
		=&\frac{6 \eta^2 p^2 M^2}{\delta^2} \mathbb{E}_{v_{[k]}}[\|\tilde{\phi}_{k-1}\|^2] +24p^2M^2 ,
		\end{align}
		where $(s.1)$ holds by adding and minus $\Phi(w_{k-1} + \delta v_k)$, $(s.2)$ holds due to Assumption \ref{ass4} and the dependency of $v_k$ with respect to $w_k$, and $(s.3)$ follows the fact that (\ref{a}) holds and $\|v_k\|=1$.	
		\begin{align}\label{add2}
		\textcircled{2} = & \frac{3p^2}{\delta^2} (\mathbb{E}_{v_{[k]}}[\|v_k\|^2 |e_{\Phi}(x^a_k,\theta_k)|^{2}] \nonumber \\ &+\mathbb{E}_{v_{[k]}}[\|v_k\|^2]  \mathbb{E}_{v_{[k]}}[|e_{\Phi}(x^a_{k-1},\theta_{k-1})|^{2}])\nonumber \\
		\overset{(s.1)}{\leq} & \frac{3p^2}{\delta^2} (\mathbb{E}_{v_{[k]}}[\|v_k\|^4])^{\frac{1}{2}} (\mathbb{E}_{v_{[k]}}[|e_{\Phi}(x^a_k,\theta_k)|^4])^{\frac{1}{2}} \nonumber \\
		&+ \frac{3p^2}{\delta^2} \mathbb{E}_{v_{[k]}}[|e_{\Phi}(x^a_{k-1},\theta_{k-1})|^{2}]\nonumber \\
		\overset{(s.2)}{\leq} & \frac{3 \mu p^2 M_{y}^{2} M_{g}^{2} }{2 \alpha_2 \delta^2} (\mathbb{E}_{v_{[k]}}[V(x^a_k,u_k,\theta_k)] \nonumber \\
		&+\mathbb{E}_{v_{[k]}}[V(x^a_{k-1},u_{k-1},\theta_{k-1})]
		), 
		\end{align}
		where $(s.1)$ holds based on the Cauchy-Schwarz inequality capable of splitting the product of two correlated random variables $v_k$ and $e_{\Phi}(x_k,\theta_k)$ and $(s.2)$ holds based on (\ref{14}).  
		Combined with the above results, the proof is completed.

\subsection{Proof of Theorem \ref{optimality-gap}}\label{E}
Since the objective function $\Phi(\theta_k)$ is convex, the Gaussian smooth approximation of $\Phi(\theta_k)$ is also convex\cite{liu2018zeroth}. With (\ref{10b}), then we have
     	\begin{align}\label{25}
     	\Phi(\theta_k) - \Phi(\theta_k^*) \leq \Phi_{\delta}(\theta_k) - \Phi_{\delta}(\theta_k^*) + 2 M \delta.
     	\end{align} 
     	With (\ref{10c}), the Taylor expansion of $\Phi_{\delta}(\theta_k)$ at solution $\theta_k^{*}$ is shown as
     	\begin{align}\label{add26}
     	\Phi_{\delta}(\theta_k) \leq & \Phi_{\delta}(\theta_k^*) + \nabla \Phi_{\delta}(\theta_k^*)^{\mathrm{T}} (\theta_k - \theta_k^{*})\nonumber\\
     	+& \frac{ M^2 p^2}{2 \delta^2}\|\theta_k - \theta_k^{*}\|^2,
     	\end{align}
     	where $\theta_{k}^*$ is the optimal solution of the problem $\mathbf{\mathcal{P}_1}$ at iteration $k$.
     	Taking the expectation of $v_{[k]}$ at both ends of the inequality (\ref{add26}), then we have
     	\begin{align*}
     	\mathbb{E}_{v_{[k]}}& [\Phi_{\delta}(\theta_k)] -\mathbb{E}_{v_{[k]}} [\Phi_{\delta}(\theta_k^{*})] \leq \nonumber \\
     	&\mathbb{E}_{v_{[k]}} [\nabla \Phi_{\delta}(\theta_k^*)^{\mathrm{T}} (\theta_k - \theta_k^{*})] 
     	+ \frac{ M^2 p^2}{2 \delta^2} \mathbb{E}_{v_{[k]}} [\|\theta_k - \theta_k^{*}\|^2].
     	\end{align*}
     	Since 
     	\begin{align*}
     	\mathbb{E}_{v_{[k]}} &[\nabla \Phi_{\delta}(\theta_k^*)^{\mathrm{T}} (\theta_k - \theta_k^{*})] \leq \nonumber \\
     	& \frac{1}{2} (\mathbb{E}_{v_{[k]}}  [\|\nabla \Phi_{\delta}(\theta_k^*)\|^2] + \mathbb{E}_{v_{[k]}} [\|\theta_k - \theta_k^{*}\|^2])
     	\end{align*}
     	where the inequality follows the fact that for $\forall a_1, a_2$,
     	\begin{align*}
     	\mathbb{E}[a_1^{\mathrm{T}} a_2] \leq (\mathbb{E}[\|a_1\|^2] \mathbb{E}[\|a_2\|^2])^{\frac{1}{2}} \leq \frac{1}{2} (\mathbb{E}[\|a_1\|^2] + \mathbb{E}[\|a_2\|^2]),
     	\end{align*} 
     	then it can be inferred that 
     	\begin{align*}
     	\mathbb{E}_{v_{[k]}} [\Phi_{\delta}(\theta_k)] &\leq \mathbb{E}_{v_{[k]}} [\Phi_{\delta}(\theta_k^{*})] + \underbrace{\frac{1}{2} \mathbb{E}_{v_{[k]}} [\|\nabla \Phi_{\delta}(\theta_k^{*})\|^2]}_{\textcircled{1}} \nonumber\\
     	&+ \underbrace{(\frac{1}{2} + \frac{ M^2 p^2}{2 \delta^2})\mathbb{E}_{v_{[k]}} [\|\theta_k - \theta_k^{*}\|^2]}_{\textcircled{2}}.
     	\end{align*}
     	
     	Next, we analyze the upper bound of the item $\textcircled{1}$ and $\textcircled{2}$.
     	\begin{align*}
     	\textcircled{1} =& \frac{1}{2} \mathbb{E}_{v_{[k]}} [\| \nabla \Phi_{\delta} (w_k^* + \delta v_k)\|^2],\nonumber \\
     	=& \frac{1}{2} \mathbb{E}_{v_{[k]}} [\|\nabla \Phi_{\delta} (w_k + \delta v_k)-(\nabla \Phi_{\delta} (w_k + \delta v_k)\nonumber\\
     	-&\nabla \Phi_{\delta} (w_k^* + \delta v_k))\|^2], \nonumber \\
     	\overset{(s.1)}{\leq} & \mathbb{E}_{v_{[k]}} [\|\nabla \Phi_{\delta} (w_k + \delta v_k)\|^2] + \mathbb{E}_{v_{[k]}} [\|\nabla \Phi_{\delta} (w_k + \delta v_k)\nonumber\\
     	-&\nabla \Phi_{\delta} (w_k^* + \delta v_k)\|^2], \nonumber \\
     	\overset{(s.2)}{\leq} &  \mathbb{E}_{v_{[k]}} [\|\nabla \Phi_{\delta} (\theta_k)\|^2] +  \frac{M^2 p^2}{\delta^2} \mathbb{E}_{v_{[k]}} [\|\theta_k-\theta_k^*\|^2],\nonumber \\
     	=&  \mathbb{E}_{v_{[k]}} [\|\nabla \Phi_{\delta} (w_k)-(\nabla \Phi_{\delta} (w_k)- \nabla \Phi_{\delta} (\theta_k))\|^2] \nonumber\\
     	+&  \frac{M^2 p^2}{\delta^2} \mathbb{E}_{v_{[k]}} [\|\theta_k-\theta_k^*\|^2],\nonumber \\
     	\overset{(s.3)}{\leq} & 2 \mathbb{E}_{v_{[k]}} [\| \nabla \Phi_{\delta} (w_k) \|^2]+ 2 M^2 p^2 \mathbb{E}_{v_{[k]}} [\|v_k\|^2] \nonumber\\
     	+ &\frac{M^2 p^2}{\delta^2} \mathbb{E}_{v_{[k]}} [\|\theta_k-\theta_k^*\|^2], \nonumber \\
     	=& 2 \mathbb{E}_{v_{[k]}} [\| \nabla \Phi_{\delta} (w_k) \|^2]+ 2 M^2 p^2 
     	+ \frac{M^2 p^2}{\delta^2} \mathbb{E}_{v_{[k]}} [\|\theta_k-\theta_k^*\|^2],
     	\end{align*}
     	where $(s.1)$ follows the fact that $\|b\|^2=\|a-(a-b)\|^2 \leq 2\|a\|^2 + 2\|a-b\|^2$, $(s.2)$ follows from (\ref{10c}), i.e., $\Phi_\delta(\theta_k)$ is $\frac{Mp}{\delta}-$ smoothness, and $(s.3)$ follows from (\ref{10c}) and $\|\delta v_k\|^2=\delta^2 \|v_k\|^2$.     	   	
     	\begin{align*}
     	\textcircled{2} =&  (\frac{1}{2} + \frac{ M^2 p^2}{2\delta^2}) \mathbb{E}_{v_{[k]}} [\|w_k - w_k^{*}\|^2],\nonumber \\ 
     	\overset{(s.1)}{\leq} & (\frac{1}{2} + \frac{ M^2 p^2}{2\delta^2}) (2 \mathbb{E}_{v_{[k]}} [\|w_{k-1}-w_{k}^*\|^2] + 2 \eta^{2} \mathbb{E}_{v_{[k]}} [ \|\tilde{\phi}_{k-1}\|^2]), \nonumber \\
     	\overset{(s.2)}{\leq} & (\frac{1}{2} + \frac{ M^2 p^2}{2\delta^2}) (2\mathbb{E}_{v_{[k]}} [\|w_{k-1}-w_{k}\|^2] + 2\eta^2 \mathbb{E}_{v_{[k]}} [ \|\tilde{\phi}_{k-1}\|^2]), \nonumber \\
     	\overset{(s.3)}{\leq} & (\frac{1}{2} + \frac{ M^2 p^2}{2\delta^2}) 4\eta^2 \mathbb{E}_{v_{[k]}} [ \|\tilde{\phi}_{k-1}\|^2], \nonumber \\
     	=& (\frac{2 \delta^2 + 2 M^2 p^2 }{\delta^2}) \eta^2 \mathbb{E}_{v_{[k]}} [ \|\tilde{\phi}_{k-1}\|^2],
     	\end{align*}
     	where $(s.1)$ follows from (\ref{aa}), $(s.2)$ follows that $\|w_{k-1}-w_k^*\|^2\leq \|w_{k-1}-w_{k}\|^2$, and $(s.3)$ follows from (\ref{a}). 
     	
     	The second moment of the gradient of $\Phi_{\delta} (w_k)$ at solution $w_k$ is $\|\nabla \Phi_{\delta} (w_k)\|^2$ and we have  	
     	\begin{align*}
     	&\mathbb{E}_{v_{[k]}} [\|\nabla \Phi_{\delta} (w_k)\|^2]\nonumber\\
     	&=  \mathbb{E}_{v_{[k]}} [\|\tilde{\phi}_k - (\tilde{\phi}_k - \nabla \Phi_{\delta} (w_k)) \|^2], \nonumber \\
     	&\leq  2 \mathbb{E}_{v_{[k]}} [\|\tilde{\phi}_k\|^2] + 2 \mathbb{E}_{v_{[k]}} [\|\tilde{\phi}_k - \nabla \Phi_{\delta} (w_k)\|^2],
     	\end{align*} 
     	where the inequality follows the fact that $\mathbb{E}[(a-b)^2] \leq 2(\mathbb{E}[a^2] + \mathbb{E}[b^2])$. Since
     	\begin{align*}
     	\mathbb{E}_{v_{[k]}}[\|\tilde{\phi}_k - \nabla \Phi_{\delta}(w_k)\|^2]\leq \mathbb{E}_{v_{[k]}} [\|\tilde{\phi}_k\|^2],
     	\end{align*}
     	which follows from $(58)$ in \cite[Theorem 8]{he2022model}, with (\ref{add18}),
     	\begin{align*}
     	\mathbb{E}_{v_{[k]}} [\|\nabla \Phi_{\delta} (w_k)\|^2] &\leq  \frac{24 \eta^2 p^2 M^2}{\delta^2}  \mathbb{E}_{v_{[k]}}[\|\tilde{\phi}_{k-1}\|^2] + 96p^2M^2 \nonumber \\
     	&+ \frac{6 \mu p^2 M_{y}^{2} M_{g}^{2} }{ \alpha_2 \delta^2} (\mathbb{E}_{v_{[k]}}[V(x^a_k,u_k,\theta_k)] \nonumber \\
     	&+\mathbb{E}_{v_{[k]}}[V(x^a_{k-1},u_{k-1},\theta_{k-1})]).
     	\end{align*}
     	   	
     	Rearranging the above items, thus we have
     	\begin{align*}
     	\mathbb{E}_{v_{[k]}} &[\Phi_{\delta}(\theta_k)] -\mathbb{E}_{v_{[k]}} [\Phi_{\delta}(\theta_k^{*})]  \nonumber \\
     	& \leq (2+ \frac{54 M^2 p^2}{\delta^2}) \eta^2 \mathbb{E}_{v_{[k]}} [ \|\tilde{\phi}_{k-1}\|^2] + 194 M^2 p^2 \nonumber \\
     	& +\frac{12 \mu p^2 M_{y}^{2} M_{g}^{2} }{ \alpha_2 \delta^2} (\mathbb{E}_{v_{[k]}}[V(x^a_k,u_k,\theta_k)] \nonumber \\
     	&+\mathbb{E}_{v_{[k]}}[V(x^a_{k-1},u_{k-1},\theta_{k-1})]).
     	\end{align*}
     	Then, it follows that
     	\begin{align}\label{26}
     	&\frac{1}{T} \sum_{k=1}^{T} \mathbb{E}_{v_{[T]}}[\Phi_{\delta}(\theta_k) - \Phi_{\delta}(\theta_k^*) ] \nonumber\\
     	&\leq (\frac{2}{T}+ \frac{54 M^2 p^2}{\delta^2 T}) \eta^2 \sum_{k=1}^{T} \mathbb{E}_{v_{[T]}} [ \|\tilde{\phi}_{k-1}\|^2] + 194 M^2 p^2 \nonumber \\
     	&+ \frac{12 \mu p^2 M_{y}^{2} M_{g}^{2} }{ \alpha_2 \delta^2 T} \sum_{k=1}^{T} (\mathbb{E}_{v_{[T]}}[V(x^a_k,u_k,\theta_k)] \nonumber \\
     	&+\mathbb{E}_{v_{[T]}}[V(x^a_{k-1},u_{k-1},\theta_{k-1})])
     	\end{align}
     	To guarantee $|\Phi_{\delta}(w) - \Phi(w)| \leq \epsilon$, we set $\delta=\frac{\epsilon}{M}$. Combined (\ref{add17}), (\ref{25}) and (\ref{26}), we obtain			
     	\begin{align}\label{28}
     	&\frac{1}{T} \sum_{k=1}^{T} \mathbb{E}_{v_{[T]}}[\Phi(\theta_k) - \Phi(\theta_k^*) ] \nonumber\\
     	& \leq \frac{\eta^2 p^2}{\delta^2 T} ( \frac{2\delta^2}{p^2} + 54 M^2  + 48\mu M_x^{2} M_y^{2} M_g^{2} ) \sum_{k=1}^{T} \mathbb{E}_{v_{[T]}} [\|\tilde{\phi}_{k-1}\|^2]\nonumber \\
     	&+ \frac{12 \mu(\mu+1) p^2 M_{y}^{2} M_{g}^{2} }{ \alpha_2 \delta^2 T} \sum_{k=1}^{T} \mathbb{E}_{v_{[T]}}[V(x^a_{k-1},u_{k-1},\theta_{k-1})] \nonumber \\
     	& + \frac{192 \mu p^2 M_x^{2} M_y^{2} M_g^{2}}{T} + 194M^2 p^2 + 2M \delta. 
     	\end{align}
     	Since $\mathbb{E}_{v_{[k]}}[\|\tilde{\phi}_{k}\|^2]$ and $\mathbb{E}_{v_{[k]}}[V(x^a_{k},u_{k},\theta_{k})]$ are coupled variables, we rely on \cite[Lemma 11]{he2022model}, which shows the upper bound of the partial sum of non-negative coupled series, to analyze (\ref{28}). 
     	
     	Combining (\ref{add17}) and (\ref{add18}), we can obtain a compacted form, which is shown as
     	\begin{align*}
     	& \left[ \begin{array}{c}
     	\mathbb{E}_{v_{[k]}}[\|\tilde{\phi}_k\|^2] \\
     	\mathbb{E}_{v_{[k]}}[\sqrt{\frac{p_{12}}{p_{21}}}V(x^a_k,u_k,\theta_k)]
     	\end{array}
     	\right] 
     	\preceq \nonumber \\
     	&P
     	\left[  \begin{array}{c}
     	\mathbb{E}_{v_{[k]}}[\|\tilde{\phi}_{k-1}\|^2] \\
     	\mathbb{E}_{v_{[k]}}[\sqrt{\frac{p_{12}}{p_{21}}}V(x^a_{k-1},u_{k-1},\theta_{k-1})] 
     	\end{array}
     	\right]
     	+
     	\left[ \begin{array}{c}
     	d_1\\
     	\sqrt{\frac{p_{12}}{p_{21}}}d_2
     	\end{array}
     	\right],
     	\end{align*}
     	where 
     	$P=\left[ \begin{array}{cc}
     	p_{11} & \sqrt{p_{12}p_{21}}\\
     	\sqrt{p_{12}p_{21}} & p_{22}
     	\end{array}
     	\right]$
     	with
     	\begin{align}\label{add27}
     	p_{11}=& \frac{6p^2 \eta^2}{\delta^2}(M^2 + \mu M_x^{2} M_y^{2} M_g^{2}), \nonumber\\
     	p_{12}=& \frac{3\mu p^2 M_y^{2} M_g^{2}}{2 \alpha_2 \delta^2}(1+\mu), \nonumber\\
     	p_{21}=& 4\alpha_2 \eta^2 M_x^2, \nonumber \\
     	p_{22}=& \mu,\nonumber\\
     	d_1=& 24p^2(M^2 + \mu M_x^{2} M_y^{2} M_g^{2}),\nonumber \\
     	d_2=& 16\alpha_2 \delta^2 M_x^2.
     	\end{align}
     	Then, we have
     	\begin{align}\label{27}
     	 \max \{\sum_{k=1}^{T}&\mathbb{E}_{v_{[T]}}[\|\tilde{\phi}_{k-1}\|^2], \sum_{k=1}^{T} \mathbb{E}_{v_{[T]}}[\sqrt{\frac{p_{12}}{p_{21}}}V(x^a_{k-1},u_{k-1},\theta_{k-1})]\} \nonumber \\
     	\leq &(\rho^{T} + \frac{1}{1-\rho})B_1+ \frac{T}{1-\rho}(d_1+\sqrt{\frac{p_{12}}{p_{21}}}d_2),\\
     	=& \mathcal{O}\left(\frac{T}{1-\rho}(p^2 + \mu p^2 +\frac{p \delta}{\eta})\right)
     	\end{align}
     	where $B_1= \mathbb{E}[\|\tilde{\phi}({w_1})\|^2]+\mathbb{E}[\sqrt{\frac{p_{12}}{p_{21}}}V(x^a_{1},u_{1},\theta_1)]$ and $\rho<1$ is the maximum singular value of the matrix $P$.
     	
     	By solving the characteristic equation $|\lambda I - P|=0$ with eigenvalues $\lambda$, then 
     	\begin{align}
     	\rho=&\frac{p_{11}+p_{22}}{2} + \sqrt{(\frac{p_{11}-p_{22}}{2})^2+p_{12}p_{21}}\nonumber \\
     	\leq & \frac{p_{11}+p_{22}}{2} + |\frac{p_{11}-p_{22}}{2}|+ \sqrt{p_{12}p_{21}} \nonumber \\
     	=& \max\{p_{11},p_{22}\}+ \sqrt{p_{12}p_{21}}.
     	\end{align}
     	To guarantee $\rho<1$, we need to set $\delta$ and $\eta$ such that
     	\begin{align}\label{30}
     	p_{11}+ \sqrt{p_{12}p_{21}}<1, \quad p_{22}+ \sqrt{p_{12}p_{21}}<1.
     	\end{align}		
     	Then, combined (\ref{28}) and (\ref{27}), it follows that
     	\begin{align}\label{add30}
     	&\frac{1}{T} \sum_{k=1}^{T} \mathbb{E}_{v_{[T]}}[\Phi(\theta_k) - \Phi(\theta_k^*) ] \leq l_3 +\nonumber\\
     	& (l_1 + \sqrt{\frac{p_{21}}{p_{12}}} l_2)
     	\left\{(\rho^{T} + \frac{1}{1-\rho})B_1 
     	+ \frac{T}{1-\rho}(d_1+\sqrt{\frac{p_{12}}{p_{21}}} d_2) \right\}
     	\end{align}
     	where 
     	\begin{align*}
     	l_1=& \frac{2\eta^2}{T}+ \frac{p^2 }{\delta^2 T}(54 M^2 \eta^2 + 48\mu M_x^{2} M_y^{2} M_g^{2} \eta^2), \nonumber\\
     	l_2=& \frac{12 \mu(\mu+1) p^2 M_{y}^{2} M_{g}^{2} }{ \alpha_2 \delta^2 T},\nonumber \\
     	l_3=& \frac{192 \mu p^2 M_x^{2} M_y^{2} M_g^{2}}{T} + 194M^2 p^2 + 2M \delta.
     	\end{align*}
     	Due to $\delta=\frac{\epsilon}{M}$, we set $\eta=\frac{\kappa \epsilon}{pT}$ such that $\frac{p^4 \eta^2}{\epsilon^2}$ and $\frac{p^2}{T}$ have the same order. Then, the order of (\ref{add30}) is shown as (\ref{20}).
     	The parameter $\kappa$ is set to satisfy (\ref{30}), i.e.,
     	\begin{align}\label{32}
     	\begin{split}
     	\xi_1 \kappa^2 + \xi_2 \kappa &<1, \\
     	\xi_3 + \xi_2 \kappa& <1,
     	\end{split}
     	\end{align}
     	where 
     	\begin{align*}
     	\xi_1&=\frac{6M^2(M^2 + \mu M_x^{2} M_y^{2} M_g^{2})}{T^2},\\
     	\xi_2&=\frac{M M_x M_y M_g}{T} \sqrt{6\mu(1+\mu)},\\
     	\xi_3&=\mu.
     	\end{align*}
     	The feasible range is denoted by $(0,\kappa^*)$. Based on (\ref{32}), we have 
     	\begin{align*}
     	\kappa^* =& \min \left\{ \frac{-\xi_2 + \sqrt{\xi^2_2 + 4 \xi_1} }{2 \xi_1}, \frac{1-\xi_3}{\xi_2} \right\},\\
     	=& \mathcal{O} \left(\min \left\{ \frac{T\sqrt{\mu(1+\mu)}}{\mu}, \frac{(1-\mu)T}{\sqrt{\mu(1+\mu)}} \right\} \right),\\
     	\rho=& \max\left\{ \xi_1 \kappa^2, \xi_3\right\}+ \xi_2 \kappa,\\
     	=& \mathcal{O} \left(\max \left\{ \frac{(1-\mu)^2}{1+\mu}, \mu\right\} +1-\mu\right).
     	\end{align*}
     	
\subsection{Proof of Lemma \ref{l8}}\label{F}
The analysis is similar to the proof of Appendix \ref{B}. First, similar to Lemma \ref{l1}, we have
	\begin{align}\label{add29}
    	|e_{\Phi}(x^a_k,\theta_k,d_k)|^2\leq \frac{\mu^{\prime} M^{2}_{y} M^{2}_{g}}{2 c_2} V(x^a_k,u_k,\theta_k,d_k).
    	\end{align}
Then, it follows that
\begin{align*}
		&V(x^a_k,u_k,\theta_k,d_k)\nonumber \\ 
		&\leq  2c_2 (\|x^a_k - x^a_{ss}(u_{k-1},\theta_{k-1},d_{k-1})\|^2 \nonumber \\
		&+ \|x^a_{ss}(u_{k-1},\theta_{k-1},d_{k-1})- x^a_{ss}(u_k,\theta_k,d_{k}) \|^2) \nonumber \\
		&\leq \mu^{\prime} V(x^a_{k-1},u_{k-1},\theta_{k-1}) \\
		&+ 2c_2(\|x_{ss}^{a}(u_{k-1},\theta_{k-1},d_{k-1})-x_{ss}^{a}(u_{k-1},\theta_{k-1},d_{k})\|^2\\
		&+ \|x_{ss}^{a}(u_{k-1},\theta_{k-1},d_{k})-x_{ss}^{a}(u_{k},\theta_{k},d_{k})\|^2)\nonumber \\
		&\leq \mu^{\prime} V(x^a_{k-1},u_{k-1},\theta_{k-1}) \\
		&+ 4c_2(\|x_{ss}^{a}(u_{k-1},\theta_{k-1},d_{k-1})-x_{ss}^{a}(u_{k-1},\theta_{k-1},d_{k})\|^2) \nonumber \\
		&+ 4c_2(\|x_{ss}^{a}(u_{k-1},\theta_{k-1},d_{k})-x_{ss}^{a}(u_{k},\theta_{k},d_{k})\|^2) \\
		&\leq \mu^{\prime} V(x^a_{k-1},u_{k-1},\theta_{k-1})) + 
		4c_2M_{x}^{2} \|\theta_k-\theta_{k-1}\|^2\\
		&+4c_2 (\|x_{ss}^{a}(u_{k-1},\theta_{k-1},d_{k})-x_{ss}^{a}(u_{k},\theta_{k},d_{k})\|^2).
		\end{align*}
	Based on Assumption \ref{ass7}, it can be inferred that
	\begin{align*}
	    \mathbb{E}_{d}&[(x_{ss}^a(u_{k-1},\theta_{k-1},d_{k-1})-x_{ss}^a(u_{k-1},\theta_{k-1},d_k))^2]\leq 4\sigma^2.
	\end{align*}
The upper bound of $\mathbb{E}_{v[k]}[\|\theta_k-\theta_{k-1}\|^2]$ is given as 
\begin{align*}
    \mathbb{E}_{v[k]}[\|\theta_k-\theta_{k-1}\|^2] \leq 2\eta^2 \mathbb{E}_{v_{[k]}}[\|\tilde{\phi}_{k-1}\|^2] + 8 \delta^2.
\end{align*}
Then, $\mathbb{E}_{v[k]}[V(x^a_k,u_k,\theta_k,d_k)]$ can be rewritten as \eqref{a30}.

\subsection{Proof of Lemma \ref{l9}}\label{G}
The analysis is similar to the proof of Appendix \ref{C}. Let $\Phi(\theta_k,d_k) \triangleq \Phi(\theta_k,h(u_k,\theta_k,d_k))$. With \eqref{ss11}, then we have
{\small{
\begin{align*}
		&\mathbb{E}_{v_{[k]}}[\|\tilde{\phi}_k\|^2] \nonumber \\
		&=\frac{p^2}{\delta^2} \mathbb{E}_{v_{[k]}}[\|v_k (\Phi(\theta_k,y^a_{k+1},d_k)- \Phi(\theta_{k-1},y^a_{k},d_{k-1}))\|^{2}] \nonumber\\
		&=  \frac{p^2}{\delta^2} \mathbb{E}_{v_{[k]}}[\|v_k (\Phi(\theta_k,d_k)- \Phi(\theta_{k-1},d_{k-1}) \nonumber\\
		&+ e_{\Phi}(x^a_k,\theta_k,d_k) - e_{\Phi}(x^a_{k-1},\theta_{k-1},d_{k-1}) )\|^{2}] \nonumber\\
		&\overset{(s.1)}{\leq}  \underbrace{\frac{3p^2}{\delta^2} \mathbb{E}_{v_{[k]}}[\|v_k (\Phi(\theta_k,d_k)- \Phi(\theta_{k-1},d_{k-1}))\|^{2}]}_{\textcircled{1}} \nonumber\\
		&+ \underbrace{\frac{3p^2}{\delta^2} \mathbb{E}_{v_{[k]}}[\|v_k e_{\Phi}(x^a_k,\theta_k,d_k)\|^{2}
		+ \|v_k e_{\Phi}(x^a_{k-1},\theta_{k-1},d_{k-1})\|^{2}]}_{\textcircled{2}}.
		\end{align*}}}
		Next, we provide the upper bound of the item $\textcircled{1}$ and $\textcircled{2}$, respectively.
		{\small{
		\begin{align*}
		\textcircled{1}
		\overset{(s.1)}{=}&\frac{3p^2}{\delta^2} \mathbb{E}_{v_{[k]}}[\|v_k (\Phi(\theta_k,d_k)- \Phi(\theta_k,d_{k-1}) \nonumber\\
		&+ \Phi(\theta_k,d_{k-1}) - \Phi(\theta_{k-1},d_{k-1}))\|^2] \nonumber \\
		\leq & \frac{6p^2}{\delta^2} \mathbb{E}_{v_{[k]}}[\|v_k (\Phi(\theta_k,d_k)- \Phi(\theta_k,d_{k-1}))\|^2]\nonumber \\
		&+ \frac{6p^2}{\delta^2} \mathbb{E}_{v_{[k]}}[\|v_k(\Phi(\theta_k,d_{k-1}) - \Phi(\theta_{k-1},d_{k-1}))\|^2]\nonumber \\
		\overset{(s.2)}{\leq}&\frac{6p^2}{\delta^2} (\mathbb{E}_{v_{[k]}}[|\Phi(\theta_k,d_{k-1}) - \Phi(\theta_{k-1},d_{k-1})|^4])^{\frac{1}{2}} (\mathbb{E}_{v_{[k]}}[\|v_k\|^4])^{\frac{1}{2}}\nonumber \\
		&+ \frac{6p^2}{\delta^2} \mathbb{E}_{v_{[k]}}[\|v_k(\Phi(\theta_k,d_{k-1})-\Phi(\theta_{k-1},d_{k-1}))\|^2]\nonumber \\
		\overset{(s.3)}{\leq}&\frac{24p^2 \sigma^2}{\delta^2} + \frac{6p^2}{\delta^2} \mathbb{E}_{v_{[k]}}[\|v_k(\Phi(\theta_k,d_{k-1})-\Phi(\theta_{k-1},d_{k-1}))\|^2]\nonumber \\
		\overset{(s.4)}{\leq}&
		\frac{24p^2 \sigma^2}{\delta^2}
		+\frac{12 \eta^2 p^2 M^2}{\delta^2} \mathbb{E}_{v_{[k]}}[\|\tilde{\phi}_{k-1}\|^2] +48p^2M^2,
		\end{align*}}}where $(s.1)$ holds by adding and minus $\Phi(\theta_k,d_{k-1})$, $(s.2)$ follows from the Cauchy-Schwarz inequality, $(s.3)$ follows from
		Assumption \ref{ass7} and $\|v_k\|=1$, (s.4) holds due to the same procedure as \eqref{add1} of the proof in Appendix \ref{C}. Similarly, the term $\textcircled{2}$ follows from \eqref{add2}. Based on the above inequalities, \eqref{add31} can be obtained. 
\subsection{Proof of Theorem \ref{optimality_noise_gap}}\label{H}
Following from the same procedure in Appendix \ref{E}, we have that
{\small{
\begin{align}\label{aa28}
     	&\frac{1}{T} \sum_{k=1}^{T} \mathbb{E}_{v_{[T]}}[\tilde{\Phi}(\theta_k) - \tilde{\Phi}(\theta_k^*) ] \nonumber\\
     	& \leq \frac{\eta^2 p^2}{\delta^2 T} ( \frac{2\delta^2}{p^2} + 98 M^2  + 96\mu^{\prime} M_x^{2} M_y^{2} M_g^{2} ) \sum_{k=1}^{T} \mathbb{E}_{v_{[T]}} [\|\tilde{\phi}_{k-1}\|^2]\nonumber \\
     	&+ \frac{12 \mu^{\prime}(\mu^{\prime}+1) p^2 M_{y}^{2} M_{g}^{2} }{ c_2 \delta^2 T} \sum_{k=1}^{T} \mathbb{E}_{v_{[T]}}[V(x^a_{k-1},u_{k-1},\theta_{k-1},d_{k-1})] \nonumber \\
     	&+ \frac{(192 \sigma^2+384\delta^2 M_{x}^{2})\mu^{\prime}p^2 M_{y}^{2} M_{g}^{2}}{\delta^2 T}+\frac{192p^2 \sigma^2}{\delta^2}+386p^2M^2
     	 + 2M \delta. 
     	\end{align}}}
     	
     	Combining \eqref{a30} and \eqref{add31}, \eqref{add27} can be reconstructed as
      \begin{align*}
          P^{\prime}=\left[ \begin{array}{cc}
     	p_{11}^{\prime} & \sqrt{p_{12}^{\prime}p_{21}^{\prime}}\\
     	\sqrt{p_{12}p_{21}^{\prime}} & p_{22}^{\prime}
     	\end{array}
     	\right]
      \end{align*}
     	with
     	\begin{align}\label{58}
     	p_{11}^{\prime}=& \frac{12p^2 \eta^2}{\delta^2}(M^2 + \mu^{\prime} M_x^{2} M_y^{2} M_g^{2}), \nonumber\\
     	p_{12}^{\prime}=& \frac{3\mu^{\prime} p^2 M_y^{2} M_g^{2}}{2 c_2 \delta^2}(1+\mu^{\prime}), \nonumber\\
     	p_{21}^{\prime}=& 8c_2 \eta^2 M_x^2, \nonumber \\
     	p_{22}^{\prime}=& \mu^{\prime},\nonumber\\
     	d_1^{\prime}=& 48p^2(M^2 + \mu^{\prime} M_x^{2} M_y^{2} M_g^{2})\nonumber \\
     	&+ \frac{24p^2 \sigma^2 (1+\mu^{\prime}M_y^{2} M_g^{2})+12p^2 \eta^2 \mu^{\prime}M_x^{2} M_y^{2} M_g^{2} }{\delta^2},\nonumber \\
     	d_2^{\prime}=& 32c_2 \delta^2 M_x^2+ 16c_2 \sigma^2.
     	\end{align}
     	Then, we have
		\begin{align*}\label{60}
     	 \max \{\sum_{k=1}^{T}&\mathbb{E}_{v_{[T]}}[\|\tilde{\phi}_{k}\|^2], \sum_{k=1}^{T} \mathbb{E}_{v_{[T]}}[\sqrt{\frac{p_{12}^{\prime}}{p_{21}^{\prime}}}V(x^a_{k},u_{k},\theta_{k},d_{k})]\} \nonumber \\
     	\leq &(\rho^{\prime^{T}} + \frac{1}{1-\rho^{\prime}})B_1^{\prime}+ \frac{T}{1-\rho^{\prime}}(d_1^{\prime}+\sqrt{\frac{p_{12}^{\prime}}{p_{21}^{\prime}}}d_2^{\prime}),\nonumber \\
     	=& \mathcal{O}\left(\frac{T}{1-\rho^{\prime}}(\frac{p^2 \sigma^2+p^2\eta^2}{\delta^2} +\frac{p \delta^2 +p \sigma^2}{\delta \eta})\right),
     	\end{align*}
     	where $B_1^{\prime}= \mathbb{E}[\|\tilde{\phi}({w_1})\|^2]+\mathbb{E}[\sqrt{\frac{p_{12}}{p_{21}}}V(x^a_{1},u_{1},\theta_1,d_1)]$ and $\rho^{\prime}<1$ is the maximum singular value of the matrix $P^{\prime}$.
     	
     	With \eqref{aa28} and \eqref{60}, it follows that
     	\begin{align}
     	&\frac{1}{T} \sum_{k=1}^{T} \mathbb{E}_{v_{[T]}}[\tilde{\Phi}(\theta_k) - \tilde{\Phi}(\theta_k^*) ] \leq l_3^{\prime} +\nonumber\\
     	& (l_1^{\prime} + \sqrt{\frac{p_{21}^{\prime}}{p_{12}^{\prime}}} l_2^{\prime})
     	\left\{(\rho^{\prime^{T}} + \frac{1}{1-\rho^{\prime}})B_1^{\prime}+ \frac{T}{1-\rho^{\prime}}(d_1^{\prime}+\sqrt{\frac{p_{12}^{\prime}}{p_{21}^{\prime}}}d_2^{\prime}) \right\}
     	\end{align}
     	where 
     	\begin{align*}
     	l_1^{\prime}=& \frac{2\eta^2}{T}+ \frac{p^2 \eta^2}{\delta^2 T}(98 M^2+ 96\mu^{\prime} M_x^{2} M_y^{2} M_g^{2}), \nonumber\\
     	l_2^{\prime}=& \frac{12 \mu^{\prime}(\mu^{\prime}+1) p^2 M_{y}^{2} M_{g}^{2} }{ c_2 \delta^2 T},\nonumber \\
     	l_3^{\prime}=& \frac{(192 \sigma^2+384\delta^2 M_{x}^{2})\mu^{\prime}p^2 M_{y}^{2} M_{g}^{2}}{\delta^2 T}\nonumber\\
     	&+\frac{192p^2 \sigma^2}{\delta^2}+386p^2M^2
     	 + 2M \delta.
     	\end{align*}
     	Due to $\delta=\frac{\epsilon}{M}$, we set $\eta=\frac{\kappa \sigma \epsilon}{p^2 \sqrt{T}}$ such that $\frac{p^4 \eta^2}{\epsilon^4}$ and $\frac{\sigma^2}{\epsilon^2 T}$ have the same order and $\kappa\in(0,\frac{p^2 \sqrt{T}}{\sigma \epsilon})$ such that $\eta \in (0,1)$. Then, the order of (\ref{aa28}) is shown as (\ref{aa20}).

\section*{Acknowledgments}
The authors would like to thank Zhiyu He (now pursuing Ph.D in ETH ) for early
inspiring discussions and valuable comments on this topic.

\bibliography{references}

\bibliographystyle{IEEEtran}

\textbf{Xiaoyu Luo} (S'19) received B.E. degree in the Department of Automation from Tianjin University, Tianjin, China, in 2019.
	She is currently pursuing the Ph.D. degree with the Department of Automation, Shanghai Jiao Tong University, Shanghai, China. She is a member of Intelligent of Wireless Networking and Cooperative Control group. Her research interests include fault-tolerant control in multi-agent systems, cooperative charging in energy storage system and security of cyber-physical systems.
	
	\textbf{Chrongrong Fang} received the B.Sc. degree in automation and the Ph.D. degree in control science and engineering from Zhejiang University, Hangzhou, China, in 2015 and 2020, respectively. He is currently an Assistant Professor with the Department of Automation, Shanghai Jiao Tong University, Shanghai, China. His research interests include anomaly detection and diagnosis in cyber-physical systems and cloud networks.
	
	\textbf{Jianping He} (M’15-SM’19) is currently an associate professor in the Department of Automation at Shanghai Jiao Tong University. He received the Ph.D. degree in control science and engineering from Zhejiang University, Hangzhou, China, in 2013, and had been a research fellow in the Department of Electrical and Computer Engineering at University of Victoria, Canada, from Dec. 2013 to Mar. 2017. His research interests mainly include the distributed learning, control and optimization, security and privacy in network systems.
	Dr. He serves as an Associate Editor for IEEE Trans. on Control of Network Systems, IEEE Open Journal of Vehicular Technology and KSII Trans. Internet and Information Systems. He was also a Guest Editor of IEEE TAC, IEEE TII, International Journal of Robust and Nonlinear Control, etc. He was the winner of Outstanding Thesis Award, Chinese Association of Automation, 2015. He received the best paper award from IEEE WCSP’17, the best conference paper award from IEEE PESGM’17, the finalist best student paper award from IEEE ICCA’17, and the finalist best conference paper award from IEEE VTC’20-Fall.

 \textbf{Chengcheng Zhao} received the PhD degree in control science and engineering from Zhejiang University, Hangzhou, China, in 2018. She is currently a research fellow in the Department of Electrical and Computer Engineering, University of Victoria. Her research interests include consensus and distributed optimization, distributed energy management in smart grids, vehicle platoon, and security and privacy in network systems. She received IEEE PESGM 2017 best conference papers award, and one
	of her paper was shortlisted in IEEE ICCA 2017 best student paper award finalist. She is a peer reviewer for Automatica, IEEE Transactions on Information Forensics and Security, IEEE Transactions on Industrial Electronics and etc. She was the TPC member for IEEE GLOBECOM 2017, 2018, and IEEE ICC 2018.

 \textbf{Dario Paccagnan} is an Assistant Professor and a member of the Computational Optimization Group in the Department of Computing, Imperial College London. Previously, he was a Postdoctoral Fellow with the Mechanical Engineering Department and the Center for Control, Dynamical Systems and
Computation, University of California, Santa Barbara. In 2018, Dario obtained a Ph.D. degree from the Information Technology and Electrical Engineering Department, ETH Zurich, Switzerland. He received a B.Sc. and M.Sc. in Aerospace Engineering in 2011 and 2014 from the University of Padova, Italy, and a M.Sc. in Mathematical Modelling and Computation from the Technical University
of Denmark in 2014; all with Honors. Dario was a visiting scholar at the
University of California, Santa Barbara in 2017, and at Imperial College of
London, in 2014. His interests are at the interface between control theory and
game theory, with a focus on the design of behavior-influencing mechanisms
for socio-technical systems. Applications include multiagent systems and
smart cities. Dr. Paccagnan was awarded the ETH medal, and is recipient
of the SNSF fellowship for his work in Distributed Optimization and Game
Design.

\end{document}